\documentclass[a4paper,10pt, reqno]{amsart}

\usepackage{amsthm}
\usepackage[utf8]{inputenc}
\usepackage[T2A]{fontenc}
\usepackage[svgnames]{xcolor}
\usepackage[pdftex]{graphicx}
\usepackage{
  amsfonts
 ,amsmath
 ,amssymb
 ,caption
 ,epigraph
 ,enumitem
 ,lipsum
 ,mathtools
 ,hyphenat
 ,tensor
 ,todonotes
 ,stackrel
 ,savesym
 ,setspace
 ,stmaryrd
 ,tikz-cd
 ,verbatim
 ,calligra}

\usepackage{tikz}
	\usetikzlibrary{calc}
	\usetikzlibrary{matrix,arrows}

\usepackage[all,2cell]{xy}\UseAllTwocells

\usepackage[
 ,sort
 ,numbers
]{natbib}

\usepackage[%
   hyperfootnotes=true
  ,colorlinks%
  ,citecolor=blue!40!black%
  ,linkcolor=blue!40!black%
  ]{hyperref}
\usepackage{mathabx,esint}
\restoresymbol{ABX}{second}

\usepackage{CJKutf8}

\usepackage[%
	left=3.75cm
	,right=3.75cm
	,top=4cm
	,bottom=6cm%
]{geometry}
\usepackage{xparse}

\ExplSyntaxOn
\NewDocumentCommand{\makeabbrev}{mmm}
 {
  \yoruk_makeabbrev:nnn { #1 } { #2 } { #3 }
 }

\cs_new_protected:Npn \yoruk_makeabbrev:nnn #1 #2 #3
 {
  \clist_map_inline:nn { #3 }
   {
    \cs_new_protected:cpn { #2 } { #1 { ##1 } }
   }
 }
\ExplSyntaxOff

\makeabbrev{\textbf}{b#1}{b,c,d,e,g,h,i,j,k,l,m,n,o,p,q,r,t,u,v,w,x,y,z,%
              B,C,D,E,G,H,I,J,K,L,M,N,O,P,Q,R,T,U,V,W,X,Y,Z}

\makeabbrev{\boldsymbol}{bs#1}{%
    a,b,c,d,e,f,g,h,i,j,k,l,m,n,o,p,q,r,s,t,u,v,w,x,y,z,%
    A,B,C,D,E,F,G,H,I,J,K,L,M,N,O,P,Q,R,S,T,U,V,W,X,Y,Z}

\makeabbrev{\mathsf}{sf#1}{a,b,c,d,e,f,g,h,i,j,k,l,m,n,o,p,q,r,s,t,u,v,w,x,y,z,%
                           A,B,C,D,E,F,G,H,I,J,K,L,M,N,O,P,Q,R,S,T,U,V,W,X,Y,Z}
\makeabbrev{\mathfrak}{f#1}{a,b,c,d,e,f,g,h,j,k,l,m,n,o,p,q,r,s,t,u,v,w,x,y,z,%
                             A,B,C,D,E,F,G,H,I,J,K,L,M,N,O,P,Q,R,S,T,U,V,W,X,Y,Z}
\makeabbrev{\mathcal}{c#1}{A,B,C,D,E,F,G,H,I,J,K,L,M,N,O,P,Q,R,S,T,U,V,W,X,Y,Z}
\makeabbrev{\mathbb}{s#1}{A,B,C,D,E,F,G,H,I,J,K,L,M,N,O,P,Q,R,S,T,U,V,W,X,Y,Z}

\newcommand{\mymu}{\boldsymbol{\mu}}
\newcommand{\myeta}{\boldsymbol{\eta}}
\makeatletter
\def\defthm#1#2{%
  \newtheorem{#1}{#2}[section]%
  \expandafter\def\csname #1autorefname\endcsname{#2}%
  \expandafter\let\csname c@#1\endcsname\c@theorem}
\makeatother

\theoremstyle{definition}

\numberwithin{equation}{section}
\newtheorem{theorem}{Theorem}[section]
    \defthm{corollary}{Corollary}
    \defthm{proposition}{Proposition}
    \defthm{lemma}{Lemma}
    \defthm{claim}{Claim}
    \defthm{definition}{Definition}
    \defthm{defn-prop}{Definition\hyp{}Proposition}
    \defthm{notation}{Notation}
    \defthm{remark}{Remark}
\title[Yosegi and Yoneda]{On the unicity of formal category theories\\%[.25ex] 
       \small \mdseries Yoneda structures, yosegi boxes, equipments}
\author[Di Liberti]{Ivan di Liberti$^\dag$}
\author[Loregian]{Fosco Loregian$^\ddag$}
\thanks{The first author is supported by grants GA17\hyp{}27844S and MUNI\fshyp{}A\fshyp{}1103\fshyp{}2017. The second author is supported by the Max Planck Inst. for Math. (Bonn), and wrote the present paper, in its entirety, during his stay at the Institute. Both authors warmly thank prof. Davide Bernardini, for his kind help reviewing a draft of the present paper. The reader interested in acquiring a copy of his `Notes on Yoneda structures' can write him at \href{mailto:davide.bernardini@uniroma1.it}{\sf davide.bernardini@uniroma1.it}.}
\address{
	Ivan di Liberti$^\dag$\newline
$^\dag$Department of Mathematics and Statistics\newline
Masaryk University, Faculty of Sciences\newline
Kotl\'{a}\v{r}sk\'{a} 2, 611 37 Brno, Czech Republic\newline
\href{mailto:diliberti@math.muni.cz}{\sf diliberti@math.muni.cz}\newline
}
\address{
	Fosco Loregian$^\ddag$\newline
	$\ddag$Max Planck Institute for Mathematics\newline
	Vivatsgasse 7, 53111 Bonn --- Germany\newline
	\href{mailto:flore@mpim-bonn.mpg.de}{\sf flore@mpim-bonn.mpg.de}
}

\makeatletter
\g@addto@macro\bfseries{\boldmath}
\makeatother
\renewcommand{\textbf}[1]{\text{\fontseries{b}\selectfont{\upshape #1}}}

\def\CAT{\mathsf{CAT}}
\def\Cat{\mathsf{Cat}}
\def\caat{\mathsf{cat}}
\def\Set{\mathsf{Set}}
\def\Ab{\mathsf{Ab}}

\def\[{\begin{equation}} \def\]{\end{equation}}

\newlength{\seplen}
\newlength{\sepwid}
\def\firstblank{\,\rule{\seplen}{\sepwid}\,}
\def\secondblank{\firstblank\llap{\raisebox{2pt}{\firstblank}}}

\DeclareDocumentEnvironment{enumtag}{m}{%
	\begin{enumerate}%
		[%
		label=\textsc{#1}\oldstylenums{\arabic*}),%
		ref=\textsc{#1}\oldstylenums{\arabic*}%
	]%
}%
{%
	\end{enumerate}%
}

\def\op{\text{op}}
\def\co{\text{co}}
\def\coop{\text{coop}}

\newcommand{\celtag}[2][dr]{\ar[#1,white, "#2"{black,description}]}

\def\ran{\mathrm{ran}}
\def\Ran{\mathrm{Ran}}

\def\lan{\mathrm{lan}}
\def\Lan{\mathrm{Lan}}

\def\rift{\mathrm{rift}}

\def\RIFT{\textsc{Rift}}
\def\leeft{\mathrm{lift}}
\def\Lift{\mathrm{Lift}}
\def\LIFT{\textsc{Lift}}
\def\Nat{\textsf{Nat}}

\newcommand{\psh}[2][\Set]{[{#2}^\op,#1]}

\DeclareMathOperator{\colim}{colim}

\def\jb{\smalltriangleleft}

\usepackage{turnstile}
\newcommand{\adjunct}[2]{\nsststile{#2}{#1}}

\let\xto\xrightarrow
\let\xot\xleftarrow

\def\sb{{}^\bullet\kern-.1em}
\tikzset{commutative diagrams/.cd,arrow style=tikz,diagrams={>=stealth'}}
\def\To{\Rightarrow}
\def\llambda{\lambda}

\newcommand{\Nearrow}{\rotatebox[origin=c]{45}{$\Rightarrow$}}
\newcommand{\Nwarrow}{\rotatebox[origin=c]{135}{$\Rightarrow$}}
\newcommand{\Searrow}{\rotatebox[origin=c]{-45}{$\Rightarrow$}}
\newcommand{\Swarrow}{\rotatebox[origin=c]{225}{$\Rightarrow$}}

\def\VCAT{\cV\text{-}\CAT}
\def\Spec{\text{Spec}}
\newcommand{\adjani}[2]{
\xymatrix@C=7mm{ A \ar@<4pt>[r]^{#1} \ar@{}[r]|\perp  & \ar@<4pt>[l]^{#2} B }}
\def\Prof{\mathsf{Prof}}
\def\yose{\left\{\begin{smallmatrix}
		\mathrm{yosegi} \\ \mathrm{boxes}
	\end{smallmatrix}\right\}}
\def\yone{\left\{\begin{smallmatrix}
		\mathrm{Yoneda} \\ \mathrm{structures}
	\end{smallmatrix}\right\}}
\def\equ{\left\{\begin{smallmatrix}
		\mathrm{Yoneda} \\ \mathrm{equipments}
	\end{smallmatrix}\right\}}
\def\reladjL#1{\tensor*[^{}_{#1}]{\dashv}{}}

\newcommand{\crc}{\cdot}
\def\dual{\lor}

\newcommand{\deduction}[4]{\begin{array}{c} #1 \to #2 \\ \hline #3 \to #4 \end{array}}
\newcommand{\RescaleSymbol}[2][.75]{\mathbin{\vcenter{\hbox{\scalebox{#1}{$#2$}}}}}
\def\myboxmin{\RescaleSymbol[.6]{\boxminus}}

\def\Kl{
  \mathchoice
    {\boldsymbol K\kern-.22em\raisebox{-.06em}{$\boldsymbol\ell$}}
    {\boldsymbol K\kern-.22em\raisebox{-.06em}{$\boldsymbol\ell$}}
    {\footnotesize \boldsymbol K\kern-.16em\scalebox{.75}{\raisebox{-.05em}{$\boldsymbol\ell$}}}
    {\boldsymbol K\kern-.22em\raisebox{-.06em}{$\boldsymbol\ell$}}
}

\def\pto{\mathrel{\ooalign{$\hfil\mapstochar\kern.2em\hfil$\cr$\to$}}}

\newcommand{\japanese}[2][min]{\begin{CJK}{UTF8}{#1} #2 \end{CJK}}

\newenvironment{coltext}[1][red]{\color{#1}}{\ignorespacesafterend}

\begin{document}
\begin{abstract}
 We prove an equivalence between cocomplete Yoneda structures and certain proarrow equipments on a 2-category $\mathcal K$. In order to do this, we recognize the presheaf construction of a cocomplete Yoneda structure as a relative, lax idempotent monad sending each admissible 1-cell $f :A \to B$ to an adjunction $\boldsymbol{P}_!f\dashv\boldsymbol{P}^*f$. Each cocomplete Yoneda structure   on $\mathcal K$ arises in this way from a relative lax idempotent monad `with enough adjoint 1-cells', whose domain generates the ideal of admissibles, and the Kleisli category of such a monad equips its domain with proarrows. We call these structures \emph{yosegi}. Quite often, the presheaf construction associated to a yosegi generates an \emph{ambidextrous} Yoneda structure; in such a setting there exists a fully formal version of Isbell duality.
\end{abstract}
\maketitle

\tableofcontents
% !TEX root = ../yosegi.tex
\section{Introduction}
Category theory attempted many times to axiomatize the essential features of the `presheaf construction' sending a small category $A$ to $\psh{A}$. Different facets of this construction (the operation of free cocompletion, the properties of the Yoneda embedding, Isbell duality, coend calculus and the theory of profunctors) suggest different directions for its efficient formalization; yet, even now that category theory reached its maturity, a `tautology' like the Yoneda lemma remains a profound and slightly mysterious statement. The present paper tries to unveil part of this mystery.

We shall say at the outset that our work fits into a solid track of previous research: in order to properly axiomatize the presheaf construction, the Australian school of category theory introduced the main notion we will be working with, \emph{Yoneda structures} \cite{street1978yoneda,street1981conspectus}; they give an efficient axiomatization for the properties of the Yoneda embedding, regarded as a cornerstone of category theory. A solution to the problem of which 2\hyp{}categories harbor a form of Yoneda lemma, spelled out in the language of internal fibrations, appeared earlier than (and somehow motivated) Yoneda structures in \cite{StreetFibreYoneda1974}.

Alternative approaches are also possible:
\begin{itemize}
	\item Wood's \emph{proarrow equipments} (see our \autoref{def-di-equi}, containing the original definition in \cite{wood1982abstract,wood1985proarrows,rosebrugh1988proarrows}) were introduced as an abstraction of the bicategory of profunctors (see \cite{benabou2000distributors,cattani2005profunctors,cordier1989shape} and \cite[§5]{cofriend} for a survey) on which the 2\hyp{}category $\caat$ of small categories embeds. This perspective focuses on the idea that categories are no more than multi\hyp{}object monoids, and presheaves are \emph{modules} over which these monoids act.
	\item Related to proarrow equipment is the notion of an \emph{hyperdoctrine}, introduced by Lawvere in \cite{lawvere1969adjointness} and expanded in \cite{pisani2012indexed,pisani2010logic} in view of its relation to indexed category theory.
	\item The formalism of \emph{KZ\hyp{}doctrines} (called, in a more precise taxonomy, \emph{lax idempotent 2\hyp{}monads}) axiomatizes the universal property of the presheaf construction as \emph{free cocompletion}; the notion was first introduced by Kock in \cite{kock1995monads}, and a link with Yoneda structure was first indirectly suggested by \cite{bunge}, and explicitly developed in \cite{walker}. Here, the author sketches a very clear connection between KZ\hyp{}doctrines \cite{kock1995monads} and Yoneda structures: a \emph{locally fully faithful} KZ\hyp{}doctrine yields a Yoneda structure once a compatible ideal of admissible maps is specified.
\end{itemize}
Yoneda structures, \emph{fibrational cosmoi} \cite{StreetCosmoi1974}, and proarrow equipments, all endow a 2\hyp{}category with additional information about a `calculus of bimodules' (the Yoneda embedding $y_A : A \to \psh{A}$, or rather its mate $\hom_A : A^\op\times A\to \Set$, is the identity 1-cell for $A$ in the bicategory $\Prof$ of profunctors). On the other hand, a KZ\hyp{}doctrine $\bsS$ endows a 2\hyp{}category $\cK$ with a procedure to build $\bsS$-cocomplete objects. These approaches are arguably related in an explicit way: easily, the strict 2\hyp{}category of small categories $A,B,\dots$, cocontinuous functors $f : \bsP A \to \bsP B$ between their presheaf categories and natural transformations is the strictification of the bicategory $\Prof$.

The present work builds an explicit connection between the universal property of $\psh{A}$ as free cocompletion, and the calculus of profunctors --embodied in the equivalence between presheaves on $A$ and discrete opfibrations on $A$. Under suitable assumptions (those of our \autoref{main-theorem}), it is possible to show that the three approaches yield equivalent theories: given a certain special kind of KZ\hyp{}doctrine, its Kleisli bicategory contains a calculus of profunctors (see §\ref{yose-to-equ}); given a sufficiently well-behaved proarrow equipment (see §\ref{equ-to-yone} and the definition of a \emph{Yoneda equipment} in \ref{def-di-yoda}) it is possible to recover the KZ\hyp{}doctrine it comes from.

The recent \cite{fiore2016relative}, offered us the main tool towards this equivalence result. In the paper, the authors introduce the notion of \emph{relative pseudomonads} to keep track of the fact that the presheaf construction $A\mapsto \psh{A}$ is not a monad \emph{stricto sensu}, since it can only take small categories as inputs: their Example 4.2 shows that when $\psh{\firstblank}$ is regarded as a relative pseudomonad, its \emph{Kleisli bicategory} consists of the bicategory $\Prof$ of profunctors. It is natural to conjecture that this is a particular example of a much general phenomenon involving the presheaf construction $\bsP$ of a generic Yoneda structure: the present paper answers this conjecture in the positive, showing \autoref{main-theorem}.

Following this result, every \emph{cocomplete} Yoneda structure on a 2\hyp{}category $\cK$ yields a relative pseudomonad $\bsP : \cA\to \cK$ via its presheaf construction, and the Kleisli bicategory $\Kl(\bsP)$ is a proarrow equipment; it is in fact a proarrow equipment of a certain well\hyp{}behaved kind (\autoref{yoda-equipment}), since it comes equipped with a $j$-relative adjunction between $\cA$ and the Kleisli bicategory $ \Kl(\bsP)$ ($j : \cA\subset \cK$ is the inclusion of `admissible' 1\hyp{}cells in $\cK$). Our §\ref{equ-to-yone} shows how such equipments allow one to reconstruct the Yoneda structure they come from.

Thus, it is possible to recover an equivalence between the above approaches to formal category theory, that correspond each other via the notion of a relative lax idempotent monad `with enough adjoint 1\hyp{}cells'. We found the notion rich and interesting \emph{per se}, and thus indulged to concoct a special name for such monads: a functor $\bsP$ as in \autoref{def:YE} will be called a \emph{yosegi box}, or a \emph{yosegi} for short). Overall, \autoref{main-theorem} produces a satisfying answer to \cite{281600}, clarifying to what extent Yoneda structures and proarrow equipments can be considered on the same footing.

We conceived the definition of yosegi in \autoref{def:YE} in order to prove the equivalence as a result of a chain of implications
\[
	\notag
	\begin{tikzcd}
		& {\yose} \arrow[Rightarrow,rd,"\text{§}\ref{yose-to-equ}"] &  \\
		{\yone} \arrow[Rightarrow,ru,"\text{§}\ref{yone-to-yose}"] &  & {\equ} \arrow[Rightarrow,ll,"\text{§}\ref{equ-to-yone}"]
	\end{tikzcd}
\]
(in this diagram each label denotes which subsection proves the implication).
Our \autoref{main-theorem} provides a dictionary to translate results between the two settings of Yoneda structures and equipments. Our §\ref{sec:repre} initiates a systematic study of which properties of a yosegi endow the associated equipment and the associated Yoneda structure with additional features. We show, for example, how a narrower and better\hyp{}behaved class of yosegis still covers all 2\hyp{}categories of the form $\VCAT$.

Somehow, the definition of yosegi is engineered in order to axiomatize the presheaf construction; on the other hand, it is willingly more general than those presheaf constructions appearing in the wild, where e.g. $\bsP$ is a `representable' functor $\psh[\Omega]{\firstblank}$ for some $\Omega\in\cK$; such constructions are the most natural example of Yoneda structure, and arise \cite{weber2007yoneda} when $\Omega$ classifies discrete opfibrations in $\cK$. Reasonably mild assumptions on $\cK$ ensure that $\Omega\cong \bsP (1)$ (this is true \emph{a fortiori} when $\Omega$ is a classifier in a 2\hyp{}topos \cite{weber2007yoneda}).

In the present work, we wanted to avoid making further assumptions on $\cK$, like those ensuring that the Yoneda structure is borne by a fibrational cosmos \cite{StreetCosmoi1974,street1980cosmoiof} or a 2\hyp{}topos structure on $\cK$. We could narrow our study to the yosegi giving rise to fibrational cosmoi and the \emph{good} Yoneda structures of Weber; in this respect, it is natural to wonder which Yoneda equipments turn the representing object of $\bsP$ into a classifier. We leave the question open for future development.

Instead, here we concentrate on studying \emph{ambidextrous} yosegi, where the yosegi $\bsP$ has a left adjoint $\bsQ$; this gives rise to a very rich theory, formalizing the behaviour of the presheaf construction on $\Cat$.

Ambidextrous yosegi correspond to cocomplete, \emph{ambidextrous Yoneda structure}, where both a covariant and a contravariant Yoneda embedding interact nicely: \autoref{ambidex} provides the precise definition for such a notion, that we expanded from a private communication \cite{bernardini2017yoneda} (it is our sincere hope that our work will spur the author to make this interesting survey available to the public); \autoref{2-sided-yosegi} observes that ambidextrous Yoneda structures correspond quite naturally to equipments where each 1\hyp{}cell $f :A\to B$ is sent to a pair of adjoints $B(f,1)\dashv B(1,f)$ in its Kleisli bicategory.

Like all theories where a sensible notion of duality is hard\hyp{}coded, ambidexterity is a particularly rich asset to $\cK$: 
\begin{enumtag}{ad}
	\item we can re\hyp{}enact \emph{Isbell duality} in an ambidextrous Yoneda structure. This is the main topic investigated in \autoref{isbellone} and the subsequent claims. Moreover, 
	\item every respectable category theory must contain a full\hyp{}fledged co/end calculus; in an ambidextrous yosegi we can define the coend operation $\int^B : [A^\dual\times A, B]\to B$ in §\ref{coendz}, and in the particular case where $B = \bsP 1$, define the bicategory of `$\bsP$-profunctors'. 
\end{enumtag}
This last point, in particular, visibly closes the circle of equivalences between formal category theories.

Appendices \ref{app:relmo} and \ref{app:contra} contain a survey of technical results: \autoref{app:relmo} is entirely devoted to an explicit description of the skew monoidal structure \cite{szlachanyi2012skew} on a functor category $[\cX, \cY]$ having as skew\hyp{}monoidal unit a prescribed functor $J : \cX \to \cY$: when the left extension along $J$ exists, and defines a functor $J_!$, the skew monoidal structure is defined by the rule $F\jb G := J_!F\crc G$ ($J_!F$ is an endofunctor of $\cY$, so the composition makes sense). As noted in \autoref{we-re-better-than-uustalu}, \autoref{app:relmo} offers a formal argument showing \cite[Thm. 3.1]{altenkirch2010monads} holds in fair generality for abstract 2\hyp{}categories: \emph{relative monads} can now defined to be internal monoid in this skew monoidal structure, so they're particular functors $T : \cX \to \cY$ with maps $\mymu : T\jb T \To T$ and $ \myeta : J \To T$ satisfying the monad axioms: of course, since the $\jb$ skew monoidal structure is highly non\hyp{}associative and non\hyp{}unital in full generality, we cannot expect but a feeble reminiscence of the monad properties (as noted in \autoref{aint-no-algebras}, when $T\not\cong J\jb T$ or $T\not\cong T\jb J$ the notion of $T$\hyp{}algebra loses meaning: finding the minimal assumptions under which they can be defined is the aim of \autoref{app:contra}, where we sketch the basic theory of lax idempotent relative monads and their algebras.
\section{Preliminaries.}
The present section collects the necessary amount of 1- and 2-dimensional category theory we shall employ in the following discussion, mainly for ease of future reference. The definition of pointwise and absolute extensions, relative adjunctions, and obviously the definition of Yoneda structure come from \cite{street1978yoneda} (we slightly differ from Street-Walters' terminology in that we do not mention `axiom 3$^*$'); the notion of cocomplete Yoneda structure is a rielaboration of \cite{bunge}; the definition of proarrow equipment is the original one given in \cite{wood1982abstract}, but see the subsequent \cite{wood1985proarrows,rosebrugh1988proarrows} for some variations on the original axioms. 
\subsection{Conventions}
Among different foundational conventions that one may adopt, we implicitly fix a universe ${\sf U}$, whose elements are termed \emph{sets}; \emph{small categories} have a \emph{set} of morphisms; \emph{locally small} categories are now small in a bigger universe ${\sf U}^+ \succ{\sf U}$; this is called the \emph{two\hyp{}universe convention}.

We shortly denote $\caat$ the 2\hyp{}category of small categories, $\Cat$ the 2\hyp{}category of ${\sf U}^+$\hyp{}small categories (each of whose hom is a set). Generic 1\hyp{}categories are denoted with uppercase Latin letters $C,D,\dots$ with the sole exception of `special' categories of algebraic structures, like e.g. sets, abelian groups and modules over a ring: $\Set, \Ab\dots$; generic 2\hyp{}categories are denoted in \verb|mathcal| face, $\cK, \cH,\dots$, but again \emph{specific} 2\hyp{}categories like $\caat, \Cat$ are denoted in a sans\hyp{}serif face.

Sometimes, when there is an adjunction between two 1\hyp{}cells $f,u$, we adopt $f\adjunct{\epsilon}{\eta}u$ as a compact notation to denote all at once that $f$ is left adjoint to $u$, with unit $\eta \colon 1 \To uf$ and counit $\epsilon\colon fu\To 1$. Given a 1\hyp{}cell $f$ and a 2\hyp{}cell $\alpha$, the \emph{whiskering} of $f$ and $\alpha$, i.e. the horizontal composition of $\alpha$ and the identity 2\hyp{}cell of $f$, is denoted $f * \alpha$ or $\alpha * f$ (according to the side of whiskering).
\begin{notation}[Conventions about extensions and lifts]\label{notat:liftext}
	While it is obvious that the notions of right and left extension and lift correspond each other under the involutions $(\firstblank)^\op, (\firstblank)^\co,(\firstblank)^\coop : \cK \to \cK$ (each taken with its appropriate variance), it is useful to have a diagram illustrating at once these universal constructions: the definition of Yoneda structure relies on the notion of left extension (and on a minor note, left lifting), so we will usually refer to left extensions to outline formal definitions; however, all four notions in \eqref{left-and-right-lift-and-ext} are perfectly dual: moving horizontally reverses the direction of 1\hyp{}cells, but not of 2\hyp{}cells; moving vertically but not horizontally reverses the direction of 2\hyp{}cells, but not of 1\hyp{}cells.
	\[\label{left-and-right-lift-and-ext}
		\begin{array}{|c|c|}\hline
			\xymatrix{
			A \ar@{}[dr]|(.3){\Swarrow\eta}\ar[d]_g \ar[r]^f        & B                                      \\
			C \ar@{.>}[ur]_{\lan_gf}                                & {\tiny \deduction{\lan_gf}{h}{f}{hg}}
			}
			                                                        &
			\xymatrix{
			{\tiny \deduction{\leeft_gf}{h}{f}{gh}}                 & C\ar[d]^g                              \\
			B\ar[r]_f \ar@{.>}[ur]^{\leeft_gf}                      & \ar@{}[ul]|(.3){\Nearrow\eta} A
			}                                                                                                \\ \hline
			\xymatrix{
			A \ar@{}[dr]|(.3){\Nearrow\varepsilon}\ar[d]_g \ar[r]^f & B                                      \\ C
			\ar@{.>}[ur]_{\ran_gf}                                  & {\tiny \deduction{hg}{f}{h}{\ran_gf}}
			}
			                                                        &
			\xymatrix{
			{\tiny \deduction{h}{\rift_gf}{gh}{f}}                  & C\ar[d]^g                              \\ B\ar[r]_f
			\ar@{.>}[ur]^{\rift_gf}                                 & \ar@{}[ul]|(.3){\Swarrow\varepsilon} A
			}                                                                                                \\ \hline
		\end{array}
	\]
	We say that a pair $\langle u,\eta\rangle$ \emph{exhibits} the left extension $\lan_gf$ of $f : A\to B$ along $g : A \to C$, but we indulge almost always to denote $u = \lan_gf$. As it is customary, we say that a left (resp., right) extension is \emph{pointwise} if for every object $C$ and $k : X\to C$ the diagram obtained pasting at $g$ the comma object $(g/k)$ (resp., $(k/g)$) is again a left extension \cite[5.2]{street1981conspectus}, and that it is \emph{absolute} if it is preserved by all 1\hyp{}cells; the same nomenclature applies to define pointwise and absolute lifts.
\end{notation}
\begin{definition}[Relative adjunction]\label{def:reladj}
	Let $f : A\to X,g : X\to B$ be a pair of 1\hyp{}cells in a 2\hyp{}category $\cK$, and $j : A\to B$; $f$ is a \emph{$j$-relative left adjoint} to $g$ if there is a 2\hyp{}cell $\eta : j \To gf$ such that the pair $\langle f, \eta\rangle$ exhibits the absolute left lifting of $j$ along $g$.
\end{definition}
\begin{remark}
	If $\cK = \Cat$, a relative adjunction consists equivalently of a natural isomorphism $B(ja,gx)\cong X(fa,x)$ or an isomorphism of profunctors $B(j,g)\cong X(f,1)$.
\end{remark}
We now introduce the notion of Yoneda structure and proarrow equipment (\autoref{def-di-equi}).
\begin{definition}[Yoneda structure]\label{def-di-yoda}
	A \emph{Yoneda structure} on a 2\hyp{}category $\cK$ consists of \emph{Yoneda data}:
	\begin{enumtag}{yd}
		\item \label{yd:uno} An ideal $\cJ_{\bsP}$ of 1\hyp{}cells called `admissible'; arrows in this ideal determine admissible objects in the class $\cJ_{\bsP, 0}\subseteq \cK$: $A$ is admissible if such is its identity 1\hyp{}cell.
		\item \label{yd:due} Each admissible object $A$ has a `Yoneda arrow' $y_A : A \to \bsP A$ to an object called the `presheaf object' of $A$.
		\item \label{yd:tre} every admissible morphism $f : A\to B$ with admissible domain fits into a triangle
		\[\label{triangle-data}
			\begin{tikzcd}
				{} & A\ar[d,phantom,"\chi^f\Searrow"]\ar[dr, "f"]\ar[dl, "y_A"'] & {} \\
				\bsP A & {} & \ar[ll, "{B(f,1)}"] B
			\end{tikzcd}
		\]
		filled by a 2\hyp{}cell $\chi : y_A \to B(f,f) =: B(f,1)\crc f$.
	\end{enumtag}
	These data satisfy the following properties:
	\begin{enumtag}{ya}
		\item \label{ya:uno} The pair $\langle B(f,1), \chi^f\rangle$ exhibits the left extension $\Lan_fy_A$.
		\item \label{ya:due} The pair $\langle f, \chi^f\rangle$ exhibits the absolute left lifting $\LIFT_{B(f,1)}y_A$.
		\item \label{ya:ter} The pair $\langle 1_{\bsP A}, 1_{y_A} \rangle$ exhibits the left extension $\Lan_{y_A}y_A$.
		\item \label{ya:qtr} Given a pair of composable 1\hyp{}cells $A \xto{f} B\xto{g} C$, the
		pasting of 2\hyp{}cells
		\[
			\begin{tikzcd}[column sep=large, row sep=large]
				A\ar[d, "f"']\ar[rr, "y_A"{name=yonA}] && \bsP A\\
				B \ar[r, "y_B"{name=yonB}]\ar[d, "g"'] & \bsP B\ar[ur, "\bsP f"']\\
				C\ar[ur, "{C(g,1)}"'] \ar[from=yonA, to=yonB, shorten >=2mm, shorten <=4mm, Rightarrow, "\chi^{y_B f}"] \ar[from=yonB, shorten >=4mm, shorten <=4mm, Rightarrow, "\chi^g"] \end{tikzcd}
		\]
		exhibits the extension $\Lan_{gf}y_A = C(gf,1)$.
	\end{enumtag}
\end{definition}
\begin{remark}[A few remarks on the axioms]\label{remaxioms}
	Note how axiom \ref{ya:due} asserts that $f$ is the $y_A$\hyp{}relative left adjoint (see \autoref{def:reladj}) of $B(f,1)$. Axioms 3-4 together entail that $A\mapsto \bsP A$ is a pseudofunctor, whose coherence morphisms are defined by the universal property of left extensions; for example, defining $\bsP f$ as $\Lan_{y_B\crc f}y_A = \bsP B(y_B\crc f,1)$ it follows from axiom \ref{ya:ter} that $B(y_A,1) = \bsP(1_A)\cong 1_{\bsP A}$. The intuition behind this is that the Yoneda embedding is a \emph{dense} map: the left extension along itself exists, it is pointwise and it is exhibited by the identity 1\hyp{}cell $1_{\bsP A}$.
\end{remark}
\begin{notation}
	For reasons that will appear evident in \autoref{presh-notat}, we denote the correspondence $f\mapsto \bsP B(y_B\crc f,1)$ as $\bsP^* f$.
\end{notation}
\begin{definition}[Cocomplete Yoneda Structure]\label{yoda-cocompleta}
	Each axiom of \autoref{def-di-yoda} encodes a certain facet of the Yoneda lemma; a \emph{cocomplete} Yoneda structure allows to mimic the universal property of $\bsP A$ as free cocompletion of $A$, and it provides as well a weak form of calculus of extensions. We say a Yoneda structure is cocomplete when in the triangle
	\[
		\begin{tikzcd}
			A \arrow[r, "f"] \arrow[d, "y_A"'] & \bsP B \\
			\bsP A \arrow[ur, "\lan_{y_A}g"'] &
		\end{tikzcd}
	\]
	where $f\in\cJ_{\bsP}$,
		\begin{itemize}
	\item  the 1\hyp{}cell $\lan_{y_A}f$ exists;
	\item  exhibits the left adjoint $\lan_{y_A}f \dashv \bsP B(f,1)$;
	\item $\lan_{y_A}f$ preserves $\lan_{y_C}g$ for every map $g: C \to \bsP A$. 
\end{itemize}
\end{definition}
\begin{remark}
	The existence of this Kan extension mirrors the notion of `cocompleteness with respect to a KZ\hyp{}doctrine' employed by \cite{walker}.
\end{remark}
\begin{remark}\label{nerve-real}
In a cocomplete Yoneda structure there is an analogue of the `nerve\hyp{}realization' paradigm relating the left extension of a 1\hyp{}cell along Yoneda to its nerve $B(f,1)$: this was already exploited in \cite[2.16]{accpres}. It was originally conjectured by the authors that, given the existence $\lan_{y_A}f$ in the diagram above,  it was possible to derive the adjunction $\lan_{y_A}f \dashv B(f,1)$ from the axioms of Yoneda structure, especially \ref{ya:due}. In $\Cat$ this is in fact possible because such a Kan extension is always pointwise. Apparently the same results does not seem to follow in full generality. Instead, in a cocomplete Yoneda structure \ref{ya:due} can be derived by $\lan_{y_A}f \dashv B(f,1)$ and thus is redundant.
\end{remark}
As already said, the definition of proarrow equipment that follows comes from \cite{rosebrugh1988proarrows} where it is spelled out for bicategories; for us, $\cA$ is a strict 2\hyp{}category.
\begin{definition}[Proarrow equipment]\label{def-di-equi}
	Let $p^* : \cA \to \cM$ be a functor between bicategories; $p^*$ is said to \emph{equip $\cA$ with proarrows} or to be a \emph{proarrow equipment for $\cA$} if
	\begin{enumtag}{pe}
		\item \label{pe:due} $p^*$ is locally fully faithful;
		\item \label{pe:ter} for every arrow $f\in\cA$, $p^*f$ has a right adjoint in $\cM$.
	\end{enumtag}
\end{definition}
We will adhere to the customary tradition to treat $\cA$ as a locally full sub\hyp{}bicategory of $\cM$, with the same objects.
% !TEX root = ../yosegi.tex
\section{Yosegi boxes}
% \epigraph{Lemarchand, who had been in his time a maker of singing birds, had constructed the box so that opening it tripped a musical mechanism, which began to tinkle a short rondo of sublime banality.}{C. Barker}
\label{sec:yosegi}
\subsection{The presheaf construction on $\Cat$}
In the present section we study the pair $(\Cat, \bsP)$, where $\bsP : \caat\to \Cat$ is the presheaf construction sending $A\mapsto\psh{A}$; this is defined having domain the category $\caat$ of small categories, and codomain the locally small ones; the embedding of $\caat$ into $\Cat$ will always be denoted as $j : \caat\subset\Cat$: it is an inclusion at the level of all cells.

We fix a notation that can be easily generalized to the case of a pair $(\cK,\bsP)$, where $\bsP : \cA\to \cK$ is the presheaf construction of a Yoneda structure on $\cK$. The section is designed in order to make the definition of yosegi in \autoref{def:YE} and the leading argument in the proof of \autoref{main-theorem} appear natural and notationally straightforward.
\begin{notation}[Presheaves]\label{presh-notat}
	We consider the functor $\bsP : A\mapsto \psh{A}$ as a \emph{covariant} correspondence on functors and natural transformations; more formally, $\bsP$ acts as a correspondence $\caat \to \Cat$ sending functors $f : A\to B$ to \emph{adjoint pairs} $\bsP_!f : \bsP A\leftrightarrows \bsP B : \bsP^* f$, and its action on 2\hyp{}cells is determined by our desire to privilege the left adjoint, inducing a 2\hyp{}cell $\alpha_! : \bsP_! f \To \bsP_!g$ for each $\alpha : f\To g$. Given $f : A\to B$, the functor $\bsP^*f := \bsP B(y_B\crc f,1)$ acts as pre\hyp{}composition with $f$, whereas $\bsP_! f$ is the operation of left extension along $f$. The situation is conveniently depicted in the diagram
	\[
		\begin{tikzcd}
			A \arrow[r, "f"] \arrow[d, "y_A"'] & B \arrow[d, "y_B"] \\
			\bsP A \arrow[shift right, r, "\bsP_! f"'] & \bsP B\arrow[shift right, l, "\bsP^* f"']
		\end{tikzcd}
	\]
	Such diagram is filled by an isomorphism when it is closed by the 1\hyp{}cell $\bsP_! f$ and by the cell $\chi^{y_B\crc f}$, as in \autoref{remaxioms}, when it is closed by $\bsP^*f$.
\end{notation}
\begin{definition}[Small functor and small presheaf]
	Let $X$ be a category; we call a functor $F : X^\op\to \Set$ a \emph{small presheaf} if it results as a \emph{small} colimit of representables; equivalently, $F$ is small if it exist a small subcategory $i : A\subset X$ and a legitimate presheaf $\bar F : A^\op\to \Set$ of which the functor $F$ is the left Kan extension along $i$.
\end{definition}
\begin{remark}
	The same definition applies, of course, to a functor $F : X \to Y$ between two large categories.
\end{remark}
\begin{notation}[Small presheaves]
	It turns out that the category of small presheaves on $X$ is legitimate in the same universe of $X$ (while the category of all functors $X^\op \to\Set$ isn't). Given a locally small category there is a Yoneda embedding $X \to [X^\op, \Set]_s$, having the universal property of free cocompletion of $X$ (see \cite[§3]{il-vecchio}). We explicitly record how the functor $\psh{\firstblank}_s:  \Cat \to \Cat$ acts on 1\hyp{} and 2\hyp{}cells: the universal property of $\psh{\firstblank}_s$ proved in \cite{il-vecchio} implies that in the square
	\[
		\begin{tikzcd}
			A \arrow[r, "f"] \arrow[d, "y_A"'] & B \arrow[d, "y_B"] \\
			\psh{A}_s \arrow[dotted, r] & \psh{B}_s
		\end{tikzcd}
	\]
	the dotted arrow exists (it is the Yoneda extension of $y_B\crc f$). The action of $\psh{\firstblank}_s$ on 2\hyp{}cells is uniquely determined as a consequence of this definition.
\end{notation}
\begin{remark}
	Unlike the previous case, the 1\hyp{}cell $\Lan_{y_A}(y_B \crc f)$ does not have a right adjoint; the reason is simple: such an adjoint would correspond to $\firstblank\crc f$, which however doesn't always restrict its action to small presheaves yielding a functor $\psh{B}_s \to \psh{A}_s$. Fortunately, it does if $f$ satisfies a certain condition that we call \emph{propriety}.%, we can ensure the existence of such 1\hyp{}cell.
\end{remark}
\begin{definition}[Proper functor]
	A functor $f : A \to B$ is called \emph{proper} if for each $i : \bar B\to B$ having small domain, the comma category $(i/f)$ (exists and it is) is small.
\end{definition}
In such an assumption, if $Q\in \psh{Y}_s$, we can consider the diagram
\[
	\begin{tikzcd}
		& (i/f) \ar[d, phantom, "\Swarrow"]\arrow[r] \arrow[ld, "i'"'] & E \ar[d,phantom,"\Swarrow"]\arrow[rd, "\bar Q"] \arrow[ld, "i"] &  \\
		X \arrow[r, "f"'] & Y \arrow[rr, "Q"'] & {} & \Set
	\end{tikzcd}
\]
since $Q\cong i_!\bar Q$ is pointwise, the whole diagram is still a left extension, so $f^*(Q)\cong i'_!(\bar Q\crc f')$.
\begin{remark}\label{salvaculo}
	All functors with small domain are proper (so in particular every Yoneda embedding $y_A : A \to \psh{A}$ is proper); another sufficient condition is to be fully faithful; proper maps, extended to be defined on 2\hyp{}categories, recover the notion of ideal of admissibles in \cite{street1978yoneda}.
\end{remark}
\begin{lemma}\label{sopra}
	Let $j : \caat\subset\Cat$ be the embedding of small categories in locally small ones. For every large category $X$ there is a natural isomorphism $[X^\op, \Set]_s\cong \Lan_j\bsP(X)$ (a convenient shorthand is to denote the left extension of $\bsP$ along $j$ as $j_!\bsP$; this is compatible with the notation in \autoref{app:relmo} and we will adopt it without further mention). More in particular, there is a canonical isomorphism between $\bsP\jb\bsP = j_!\bsP \cdot \bsP$ and the functor $[(\bsP \firstblank)^\op, \Set]_s$, where for a large category $X$, the category $\psh{X}_s$ designates small presheaves $X^\op \to \Set$.
\end{lemma}
\begin{proof}
	To show that the universal property of the Kan extension is fulfilled by $[X^\op, \Set]_s$ we employ a density argument: given a functor $H : \Cat\to \Cat$, every natural transformation $\alpha : \bsP \To H\crc j$ can be extended to a natural transformation $\bar\alpha$ from small presheaves to $H$, using the fact that each $F\in \psh{X}_s$ can be presented as a small colimit of representables: the components of $\bar\alpha_X$ are defined, if $F \cong i_! \tilde F \in\psh{X}_s$ for $i : A\to X$, as
	\[\notag
		\psh{X}_s \xto{i^*} \psh{A}_s = \psh{A} \xto{\alpha_A} HjA = HA \xto{Hi} H.X\qedhere
	\]
\end{proof}
\begin{corollary}
	From this it follows that there is a canonical integral isomorphism
	\[[X^\op, \Set]_s \cong \int^{A\in\caat} [X,A]^\op\times[A,\Set]\]
	(in particular, this specific coend exists even if it is indexed over a non\hyp{}small category); this will turn out to be useful in the proof of \autoref{sotto}.
\end{corollary}
\begin{remark}\label{not-a-skiu}
	It is reasonable to expect $j_!\bsP$ to be the small\hyp{}presheaf construction; this construction is the legitimate version of the Yoneda embedding associated to a (possibly large) category $X$. It is important to stress our desire to exploit the results in \cite{fiore2016relative}, but minding that $\bsP$ has additional structure (their approach qualifies $\bsP$ as a monad, but only in the sense that it is a $j$-pointed functor, endowed with a unit $\myeta : j \to \bsP$ and with a `Kleisli extension' map --instead of a monad multiplication-- sending each $f : jA\to \bsP B$ to $f^\star : \bsP A\to \bsP B$); in view of \autoref{it-is-skiu}, now we would like to say that $[\caat, \Cat]$ is a skew\hyp{}monoidal category with skew unit $j$, and composition $(\bsF,\bsG)\mapsto j_!\bsF\crc \bsG$: this would yield `iterated presheaf constructions' $\bsP \jb\bsP$, $\bsP\jb\bsP\jb\bsP,\dots$, all seen as functors $\caat\to\Cat$. Unfortunately, given a functor $F: \caat \to \Cat$ it is impossible to ensure that $j_!F$ exists in general ($\caat$ is a ${\sf U}^+$-category, and $\Cat$ can't be ${\sf U}^+$-cocomplete), thus --if anything-- the skew\hyp{}monoidal structure of \autoref{it-is-skiu} does not exist globally. Fortunately \emph{some} left extensions --precisely those we need-- exist, so we can still employ the `local' existence of $j_!\bsP$ and its iterates to work as if it was part of a full monoidal structure.

	It also turns out (and this is by no means immediate, see \autoref{coirens}) that the unitors $\lambda_{\bsP} : j\jb\bsP \to \bsP$ and $\varrho_{\bsP} : \bsP \to \bsP\jb j$ and the associator $\gamma_{\bsP} : (\bsP\jb\bsP)\jb\bsP \to \bsP\jb(\bsP\jb\bsP)$ are all invertible.
		
	Apart from their relevance in view of our \autoref{main-theorem}, the preliminary results in this section legitimate the practice to naively consider iterated presheaf constructions: once we consider small functors, the category $[X^\op,\Set]_s$ lives in the same universe of $X$, and so do all categories $\psh{\psh{A}}, \psh{\psh{\psh{A}}}$.
\end{remark}
\begin{lemma}[$j_!\bsP$ preserves itself]\label{sotto}
	There is a canonical isomorphism
	\[
		\tilde\gamma_{\bsP\bsP} : j_!(j_!\bsP \cdot \bsP) \cong j_!\bsP\cdot j_!\bsP.
	\]%; formally, `the functor $j_!\bsP$ preserves itself'.
\end{lemma}
\begin{proof}
	Relying on the previous lemma, we compute the coend
	\[
		j_!(j_!\bsP\cdot \bsP) \cong \int^{A\in \caat} [A,X] \times [(\bsP A)^\op, \Set]_s
	\]
	which is now isomorphic to $[[X^\op,\Set]_s^\op,\Set]_s$ in view of the Yoneda reduction and of \autoref{sopra} (note that the definition of $\llambda A.\bsP A$ and $\llambda A.j_!\bsP A$ entail that $\llambda A.\psh{\bsP A}$ is covariant in $A$).
\end{proof}
\begin{lemma}\label{coirens}
	There are canonical isomorphisms $\lambda_{\bsP} : j\jb \bsP \cong\bsP$, $\varrho_{\bsP} : \bsP \cong \bsP\jb j$ and $\gamma_{\bsP\bsP\bsP} : (\bsP\jb\bsP)\jb\bsP\cong \bsP\jb(\bsP\jb\bsP)$, determined as in \ref{s:uno}--\ref{s:tre}.
\end{lemma}
\begin{proof}
	The isomorphisms come from the unit and associativity constraints of \ref{s:uno}-\ref{s:tre}; as noted in \autoref{adjoints-are-exts}, $\varrho$ is invertible in every component because $j$ is fully faithful, and similarly $\lambda$ is invertible in every component if we show that $j$ is a dense functor. Once we have shown this, \autoref{sotto} above will conclude, since from \eqref{associator} the associator $\gamma_{\bsP\bsP\bsP}$ coincides with the composition $\tilde\gamma_{\bsP\bsP} * \bsP$.

	Now, the functor $j$ is dense, because the full subcategory of $\Cat$ on the generic commutative triangle $\triangle[2]$ is dense. So $\caat$, being a full supercategory of a dense category, is dense.
\end{proof}
\begin{remark}
	We often write $\myeta\bsP$, $\bsP \myeta$, $\mymu\bsP$\dots to denote what should be written as $\myeta \jb\bsP$, $\bsP \jb\myeta$, $\mymu\jb\bsP$\dots{} Keeping in mind \autoref{on-uisge}, that clarifies what is the formal definition for $\jb$-whiskerings, there is non chance of confusion: $\myeta \jb \bsP = j_!\myeta * \bsP$, $\bsP \jb \myeta = j_!\bsP * \myeta$, and similarly for every other whiskering.
\end{remark}
We now would like to prove the following result:
\begin{proposition}\label{set-is-a-ianua}
	The presheaf construction $\bsP = \psh{\firstblank}$ is a lax idempotent $j$-relative monad.
\end{proposition}
In order to do so, we have to define what lax idempotent relative monads are: giving a self\hyp{}contained survey of this matter is the purpose of \autoref{app:relmo} and \autoref{app:contra}. As already said, we only slightly detour from the approach in \cite{fiore2016relative}:%, because since all the left extensions we need exist (and have known universal properties, as shown above in \autoref{sopra}) allows us to ignore the non\hyp{}existence of a global skew\hyp{}monoidal structure on the whole $[\caat,\Cat]$, and to reason `locally'.%and the reason why we employ a notation closer to the existing literature.% but more general, as it `relativizes' the monads.

	A relative monad $T : \cX\to \cY$ is the formal equivalent of a monad in the skew\hyp{}monoidal structure $([\cX,\cY],\jb)$; on $[\caat,\Cat]$, such structure exists `locally' in the components we need allowing us to  work as if $\bsP$ really was a $\jb$-monoid. As already said, \cite{fiore2016relative} does not assume the existence of a multiplication map $\mymu : \bsP\jb\bsP \to \bsP$, replacing it with coherently assigned `Kleisli extensions' to maps $f : jA\to \bsP B$.

	A relative monad is now \emph{lax idempotent} (or a \emph{KZ\hyp{}doctrine}) if it satisfies the 2\hyp{}dimensional analogue of the notion of idempotency; in short, when the algebra structure on an object $A$ is unique up to isomorphism as soon as it exists. Following \cite[2.2]{GARNER20121372}, $\bsP$-algebras as cocomplete categories, thus, when an object is a $\bsP$-algebra it is so in a unique way (this is of course a behaviour of all formal cocompletion monads).
\begin{proposition}[$\bsP$ is a yosegi, I]\label{set-is-ian:uno}
	$\bsP$ is a $j$-relative monad if (as always) $j : \caat \subset \Cat$ is the obvious inclusion.
\end{proposition}
\begin{proof}
	As already noted, the relevant left extensions involved in this proof exist; now, the diagrams we have to check the commutativity of are the following, once we define the Yoneda embedding $y_A : A \to \bsP A$ as unit, and $\bsP^* y_A : \bsP\bsP A\to \bsP A$ (it exists because \autoref{salvaculo}) as multiplication of the desired monad.
	\begin{center}
		\includegraphics[width=\textwidth]{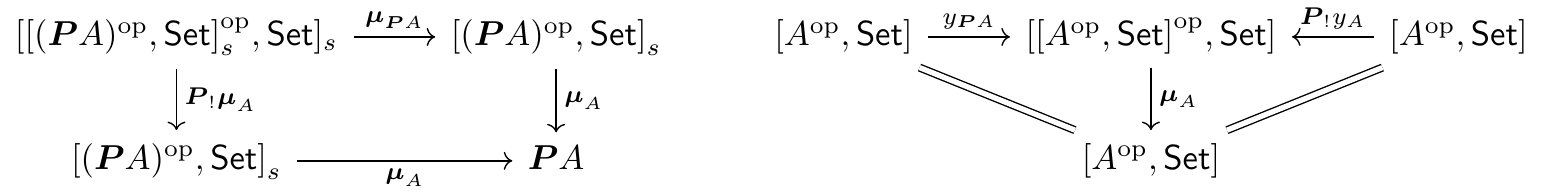}
	\end{center}
	In order to show that they commute, we will exploit the adjunctions $\bsP_!y_A\dashv \bsP^* y_A$ and $\bsP_!\mymu_A \dashv \bsP^*\mymu_A$ (the functor $\bsP^*\mymu_A$ exists because it must coincide with the Yoneda embedding of $\psh{\bsP A}_s$ into $\psh{\psh{\bsP A}_s}_s$; it must act as the Yoneda embedding $Q\mapsto \psh{\bsP A}_s(\firstblank,Q)$, and this evidently lands into the category of small presheaves on $\psh{\bsP A}_s$ when restricted to small functors).

	For what concerns the unit axiom, the commutativity of the left triangle can be deduced from the chain of isomorphisms
	\begin{align*}
		\mymu_{A} \cdot y_{\bsP A}   & \cong   \Lan_{y_{\bsP A} y_A}(y_A) \cdot y_{\bsP A}       \\
		                             & \cong \Lan_{y_{\bsP A}}(\Lan_{y_A}(y_A)) \cdot y_{\bsP A} \\
		(y_{\bsP A} \text{ is f.f.}) & \cong  \Lan_{y_A}(y_A)                                    \\
		(y_A \text{ is dense})       & \cong \text{id}_{\bsP A}.
	\end{align*}
	The right triangle corresponds to the composition
	\[\notag
		\begin{tikzcd}[row sep=0]
			\psh{A} \ar[r] & \psh{\psh{A}}_s \ar[r] & \psh{A}\\
			P \ar[r,mapsto] & (Q\mapsto \Nat(Q,P)) \ar[r,mapsto] & (a\mapsto\Nat(y_A(a),P)\cong Pa)
		\end{tikzcd}
	\]
	which is again isomorphic to the identity of $\bsP A$ thanks to the Yoneda lemma.

	In order to show that the multiplication is associative, we prove that $\bsP_!\mymu_A\cong \mymu_{\bsP A}$ as a consequence of the fact that	there is an adjunction $\mymu_{\bsP}\adjunct{1}{} \bsP \mymu$, having moreover invertible counit. The argument will be fairly explicit, building unit and counit from suitable universal properties of the presheaf construction and from the definition for $\bsP^*\mymu_A$ and $\mymu_{\bsP A}$:
	\begin{itemize}
		\item $\bsP^* \mymu_A$ sends $\zeta : \bsP A^\op \to \Set$ into $\psh{\bsP A}_s(\firstblank,\zeta)$ (it plays the exact same r\^ole of a large Yoneda embedding; this will entail that the counit of the adjunction $\mymu_{\bsP}\dashv \bsP^*\mymu$ is invertible);
		\item $\mymu_{\bsP A}$ acts sending $\llambda F.\Theta(F) \in \psh{\psh{\bsP A}_s}_s$ to $\llambda a.\Theta(\hom(\firstblank,a))$.%\footnote{We employ lambda notation, but since the Greek alphabet is already overloaded with different meaning, we employ a cyrillic \emph{el} to denote the operation of lambda abstraction.}
	\end{itemize}
	With these definitions, the composition $\mymu_{\bsP A}\crc \bsP^* \mymu_A$ is in fact isomorphic to the identity of $\psh{\bsP A}_s$; we not find the unit map: we refrain from showing the zig\hyp{}zag identities as they follow right away from the explicit description of the co/unit.

	The unit will have as components morphisms $\Theta \To \bsP^* \mymu_A(\mymu_{\bsP A}(\Theta))$ natural in $\Theta\in \psh{\psh{\bsP A}_s}_s$: given one of these components, its codomain can be rewritten as the coend
	\begin{align*}
		\bsP^* \mymu_A(\mymu_{\bsP A}(\Theta)) & = \llambda\chi.\psh{\bsP A}_s(\chi,\mymu_{\bsP A}(\Theta))                      \\
		                                       & \cong \llambda\chi.\int_{F\in\bsP A}\Set(\chi(F),\mymu_{\bsP A}(\Theta)(F))     \\
		                                       & \cong \llambda\chi.\int_{F\in\bsP A}\Set(\chi(F),\Theta(\bsP A(\firstblank,F))) \\
		                                       & \leftarrow \llambda\chi.\Theta(\chi)
	\end{align*}
	and we obtain the candidate morphism in the last line as follows: its component at $\chi \in \psh{\bsP A}_s$ must be an arrow
	\[
		\Theta(\chi) \longrightarrow  \int_{F\in\bsP A}\Set(\chi(F),\Theta(\bsP A(\firstblank,F)))
	\]
	which is induced by a wedge
	\[\Theta(\chi)\to \Set(\chi(F),\Theta(\bsP A(\firstblank,F)));
	\]
	such wedge comes from (the mate of) $\Theta^\text{ar}$, $\Theta$'s function on arrows: the Yoneda lemma now entails that such action induces a map
	\[
		\chi(F)\cong \psh{\bsP A}_s(\bsP A(\firstblank,F),\chi) \xto{\;\Theta^\text{ar}\;} \Set(\Theta(\chi),\Theta(\bsP A(\firstblank,F)))
	\] that by cartesian closure can be reported to
	\[
		\Theta(\chi)\longrightarrow \Set(\chi(F),\Theta(\bsP A(\firstblank,F)))
	\] Of course, this is a wedge in $F$, and we conclude.
\end{proof}
\begin{proposition}[$\bsP$ is a yosegi, II]\label{set-is-ian:due}
	The monad $\bsP$ is lax idempotent in the sense of \autoref{def:laxidem}: there exist an adjunction $\mymu_A\adjunct{1}{}\myeta_{\bsP A}$.
\end{proposition}
\begin{proof}
	The existence of an adjunction $\mymu_A\adjunct{1}{}\myeta_{\bsP A}$. will imply all the equivalent conditions in \autoref{lax-equivs}, that we nevertheless recall in \autoref{lax-equivs-for-P} below for the convenience of the reader. From the definition of these maps, there is a natural candidate to be the counit, and this will be invertible as a consequence of the Yoneda lemma. Indeed, the isomorphism $1 \cong \mymu_A\cdot \myeta_{\bsP A}$ corresponds to the map
	\[
		\llambda a.Fa\mapsto \llambda G.\hom(G,F)\mapsto \llambda a.\hom(A(\firstblank,a),F)\cong Fa.
	\]
	The unit is instead given by the action of a certain functor on arrows, in a similar way as above: given $\chi\in \psh{\bsP A}_s$, there is a canonical map
	\begin{align*}
		\chi(F) & \to \bsP A(F, \llambda a.\chi(\hom(\firstblank,a)))   \\
		        & \cong \int_{a\in A}\Set(Fa,\chi(\hom(\firstblank,a)))
	\end{align*}
	coming from (the mate of) a wedge
	\[
		Fa\cong \bsP A(\hom(\firstblank,a),F) \to \Set(\chi(F), \chi(\hom(\firstblank,a)))
	\]
	defined by the action on arrows of $\chi$.
\end{proof}
%\begin{remark}\label{there-s-adjoint}
%	Since $\bsP^*$ preserves the direction of the adjunctions he acts on, reversing the r\^ole of unit and counit, it is clear how applying it to the adjunction of \autoref{set-is-ian:due} we get an adjunction $\bsP^*\myeta_{\bsP A} = \mymu_{\bsP A}\dashv \bsP^*\mymu_A$ with invertible counit.
%\end{remark}
	The following remark, which is a particular case of \autoref{lax-equivs}, characterizes $\bsP$-algebras as categories whose Yoneda embedding $y_A$ has a left adjoint $\alpha$. These are the \emph{cocomplete} categories \cite[§2]{GARNER20121372}; it is rather easy to see that one of the axioms of $\bsP$-algebras asserts that $\alpha(A(\firstblank,a))\cong a$, and since $\alpha$ is a left adjoint it is uniquely determined by sending a colimit of representables into the colimit in $A$ of representing objects (all such colimits, in particular, exist).
\begin{remark}\label{lax-equivs-for-P}
	It turns out from the general theory of lax idempotent monads, sketched in \autoref{app:contra}, that $\bsP$ satisfies the following equivalent conditions:
	\begin{enumtag}{pl}
		\item \label{pl:uno} for every pair of $\bsP$-algebras $a,b$ and morphism $f :A \to B$, the square
		\[
			\notag\footnotesize
			\begin{tikzcd}[row sep=4mm,column sep=4mm]
				\bsP A \arrow[d,"a"']\ar[r,"\bsP f"] &  \bsP B \arrow[d, "b"] \\
				A \arrow[r,"f"'] & B
			\end{tikzcd}
		\]
		\item \label{pl:due} if $\alpha : \bsP A\to A$ is a $\bsP$-algebra, there is an adjunction $\alpha\dashv \myeta_A$ with invertible counit;
		\item \label{pl:ter} there is an adjunction $\mymu_A\dashv \myeta_{\bsP A}$ with invertible counit.

	\end{enumtag}
	In particular, we have shown condition \ref{pl:ter}; since \ref{pl:due} holds, there is only a possible choice up to isomorphism for a $\bsP$-algebra structure on an object $A$, namely the arrow playing the r\^ole of left adjoint to the Yoneda embedding.
\end{remark}
\subsection{Yosegi boxes}
In the previous section we showed that the functor $\bsP_! : \caat \to \Cat$ defining the presheaf construction in \autoref{presh-notat} enjoys the following properties:
\begin{enumtag}{yb}
	\item \label{yb:uno} every $\bsP_!(f)$ fits into an adjunction $\bsP_!(f)\dashv \bsP^*(f)$;
	\item \label{yb:due} $\bsP_!$ is a relative monad with respect to the inclusion $j : \caat\subset\Cat$;
	\item \label{yb:tre} the monad $\bsP_!$ is lax idempotent.
\end{enumtag}
\begin{definition}[Yosegi box]\label{def:YE}
	Let now $j: \cA \subset \cK$ be the inclusion of a full sub\hyp{}2\hyp{}category of $\cK$: every 2\hyp{}functor $\bsP : \cA \to \cK$ satisfying the same properties \ref{yb:uno}, \ref{yb:due}, \ref{yb:tre}, and such that moreover
	\begin{enumtag}{}
		\item[\textsc{yb}\oldstylenums{0})] the monad $\bsP_!$ has fully faithful unit $\myeta_A$;
		\item[\textsc{yb} -\oldstylenums{1})] The 1\hyp{}cell $\bsP^*\myeta_A$ exists ( its existence is not implied by \ref{yb:uno} because $\myeta_A$ does not lie in $\cA$ in general. Strictly speaking we are saying that the right adjoint to $j_! \bsP_!(\myeta_A)$ exists).
	\end{enumtag}
	will be called a \emph{yosegi box.}\footnote{\emph{Yosegi-zaiku} (\japanese{寄木細工}) is a kind of Japanese marquetry featuring elaborate inlaid and mosaic designs; the amount of data described in \autoref{def:YE} can be thought as a set of tightly linked properties, each of which is rich of adornments and properties.}
\end{definition}
\begin{remark}[Applying a yosegi outside its domain] Whenever we write $\bsP C$, but such an object is not defined (because $C$ lies in $\cK\smallsetminus \cA$), we mean $j_! \bsP C$. This slight abuse of notation makes certain proofs more readable; the tacit convention to allow only the letter $A$ denote an admissible object of a Yoneda structure, or an object in the domain of a yosegi $\bsP$, shuts down many possible sources of confusion.
\end{remark}
\begin{remark}[Nervous $1$-cells]\label{nervous} 
Let $f: A \to B$ be a $1$-cell. When either its domain or codomain does not lie in $\cA$ the map $\bsP^* f$ is not guaranteed to exist by the axiom \ref{yb:uno}. The $1$-cells for which it exists and for all $g: A \to \bsP C$ the \textit{cancellation rule} $$\lan_{\myeta_A}g \circ \bsP^* f \circ \myeta_B \cong \lan_gf $$ holds will be called \textit{nervous} maps. This name is justified by the existence of the nerve construction: in a yosegi the 1-cell $B(f,1)$ exists if and only if  $\bsP^* f$ does, as in such case $\bsP^* f\cong \Lan_{\myeta_B}(B(f,1))$. When $\bsP^* f$ exists, $B(f,1)$ is defined as $\bsP^* f \circ \myeta_B$. To the best of our knowledge, the name \textit{nervous} appeared for the first time in \cite{bourke2018monads}, where the notion of \emph{nervous monad} is studied with a very similar spirit. The concept itself goes back to \cite{bunge} where nervous maps are called \textit{admissible}. \cite{walker}  exploited this notion for the same purpose of our §\ref{equ-to-yone}.
\end{remark}
\begin{remark}[Cancellation rule]\label{crule} 
A compressed and operative version of the cancellation rule above, that will be useful later in the text, is that in the span $\bsP C \xot{g}A\xto{f}  B$ where $f$ is nervous, $\lan_fg$ exists and coincides with $\lan_{\myeta_A}(g) \circ \lan_f(\myeta_A)$.
\end{remark}
\begin{remark}[Yosegi boxes and hyperdoctrines]
	The definition of yosegi box resembles in many ways the definition of \emph{hyperdoctrine} due to Lawvere \cite{lawvere1969adjointness}. There, the author chooses as codomain of a certain 2\hyp{}functor a prescribed family of categories, usually having no connection with $\cA$; this resemblance is not surprising, as \cite{pisani2010logic,pisani2012indexed} shows how a considerable amount of formal category theory can be recovered in the language of hyperdoctrines.
\end{remark}
\subsection{Formal Kan Lemma}
We want to provide a formal analogue of the following well\hyp{}known and very useful result:
\begin{lemma}[Kan lemma]\label{kanlemma}
	Given a span $C \xot{g}A\xto{f} B$ of categories, where $B$ is cocomplete and $A$ is small, the following left extension exists:
	\[
		\begin{tikzcd}
			A \arrow[r, "f"] \arrow[d, "g"'] & B \\
			C \arrow[ur, "\lan_g(f)"'] &
		\end{tikzcd}
	\]
\end{lemma}
This result appears for example in \cite[3.7.2]{Bor1}, and it was originally due to Kan. 

A yosegi box is a sufficiently rich structure to re\hyp{}enact a formal analogue of this theorem. As little as this result may appear, this serves as a first motivation for \autoref{def:YE}.
\begin{lemma}[Formal Kan lemma]
	Let $\cK$ be a 2\hyp{}category, and $\bsP : \cA \to \cK$ a yosegi box; given a span $C \xot{g}A\xto{f} B$ of 1\hyp{}cells in $\cK$, where $A\in \cA$, then if in the diagram
	\[
		\begin{tikzcd}
			A \arrow[r, "f"] \arrow[d, "g"'] & B \\
			C \arrow[ru, "\lan_{g}(f)"', dotted]  &
		\end{tikzcd}
	\]
	the object $B$ has a structure of $\bsP_!$-algebra and $g$ is nervous, then the dotted left extension $\lan_gf$ exists.
\end{lemma}
\begin{proof}
	The following argument essentially reconstructs \cite[3.7.2]{Bor1}; consider the diagram
	\[
		\begin{tikzcd}
			A \ar[rr,"f"]\ar[dd,"g"']&& B\ar[dd,shift left=1mm, "\myeta_B"] \\
			\\
			 C \ar[rr,"\lan_g(\myeta_Bf)"] && \bsP B \ar[uu, shift left=1mm, "L_B"]\\
		\end{tikzcd}
	\]
	Where $\lan_g(\myeta_Bf)$ exists because $g$ is nervous (see \autoref{crule}), and $\bsP B$ is a shorthand for $j_!\bsP B$. We claim that the  dotted arrow coincides with the composition $L_B \crc \lan_g(\myeta_Bf)$ and together with a naturally determined 2\hyp{}cell exhibits the desired left extension. Recall  that the 1\hyp{}cell $L_B : \bsP B \to B$ is precisely what gives $B$ the structure of a $\bsP_!$-algebra, and that this is exactly what happens for $\bsP = \psh{\firstblank}$ (as already noted, $\bsP_!$-algebras for the presheaf construction of $\Cat$ are the cocomplete categories). The claim follows from the equalities
	\[
		L_B \crc \lan_g(\myeta_Bf) \cong \lan_{\myeta_B}(\text{id}_B) \crc \lan_g(\myeta_Bf) \stackrel{\ref{crule}}{\cong} \lan_g(f). \qedhere
		\]
\end{proof}
\subsection{The adjoint functor theorem}
The adjoint functor theorem represents one of the most applicable results in basic category theory. Any respectable formal category theory should not be exempt from providing a way to ensure that a colimit preserving map is a left adjoint; we present here a formal version of this theorem, maybe equivalent, but apparently more manageable than the one in \cite{street1978yoneda}. 
\begin{theorem}[Formal adjoint functor theorem]\label{AFT}
Let $f: A \to B$ be a $1$-cell of $\bsP_!$-algebras; then the following are equivalent.
\begin{enumerate}
\item $f$ is a left adjoint;
\item $f$ is a nervous map and a morphism of $\bsP_!$-algebras.
\end{enumerate}
\end{theorem}
\begin{remark}
Being a morphism of $\bsP_!$-algebras, in the context of ordinary categories and for the classical presheaf construction, means that $f$ preserves colimits, while being a nervous map (\autoref{nervous}) corresponds to a very sharp version of the solution set condition.
\end{remark}
\begin{proof}%[Proof of Theorem \ref{AFT}]
Let's start by $2 \Rightarrow 1)$. By the Formal Kan lemma \autoref{kanlemma}, $\lan_f(1)$ exists (and we can explicitly construct it). Since $f$ is a morphism of $\bsP_!$-algebras, it preserves $\Lan_f(1)$, in fact:

	\begin{align*}
		f \crc \lan_f(1)   & \cong   	f \crc \lan_f(L_A\myeta_A)       \\
		                             & \cong  	f \crc L_A \crc  \lan_f(\myeta_A)  \\
		(f \text{ is a morphism of algebras}) & \cong  \lan_{y_A}(f)    \crc  \lan_f(\myeta_A)                       \\
		(\ref{crule})       & \cong  \lan_ff.
	\end{align*}
 Now (see for example \cite[Prop. 2]{street1978yoneda}) this is equivalent to the fact that $f$ has $\Lan_f1$ as right adjoint. In order to prove the converse implication note that when $f$ has a right adjoint $r$, the map $\myeta_A \circ r$ has the universal property of $B(f,1)$; this implies that $f$ is a nervous map. The proof that $f$ is a morphism of algebras is a reformulation of \cite[Prop. 15]{street1978yoneda}.
\end{proof}

% !TEX root = ../yosegi.tex
\section{Yosegi boxes are Yoneda structures are equipments}
\label{sec:equipments}
We introduce the notion of Yoneda equipment as a refinement of \autoref{def-di-equi}; strengthening the original definition of Wood \cite{wood1982abstract} allows to recover the Yoneda structure, or rather the yosegi, of which the equipment is secretly the free functor in a Kleisli construction.
% We recall the definition of proarrow equipment and its refinements in \autoref{def-di-equi}. Here we propose another refinement of Wood's idea, capable to capture the notion of Yoneda structure.
\begin{definition}[Yoneda equipment]\label{yoda-equipment}
	In a triangle of 2\hyp{}functors
	\[
		\begin{tikzcd}
			& \cA \ar[d,phantom, "\To"]\arrow[rd, "p"] \arrow[ld, "j"'] &  \\
			\cK  & {} & \cM \ar[ll, "u"]
		\end{tikzcd}
	\]
	the inclusion $j : \cA\subset \cK$ is said to admit a \emph{Yoneda equipment} $p \reladjL{j} u$ if
	\begin{itemize}
		\item $p$ is a proarrow equipment \emph{à la Wood}, in the sense of \autoref{def-di-equi};
		\item $p$ is the $j$-relative left adjoint of $u : \cM\to \cK$;
		\item the relative monad $up$, having unit $\myeta : j\To up$ is lax idempotent in the sense of \autoref{def:laxidem} and such unit is fully faithful;
	\end{itemize}
\end{definition}
\begin{remark}
	The inclusion $j : \caat\subset \Cat$ admits a Yoneda equipment via the canonical equipment of $\caat$ with proarrows given by the bicategory $\Prof$ of profunctors
	\[
		\begin{tikzcd}
			& \caat \arrow[rd, "p"] \arrow[ld, "j"'] &  \\
			\Cat &  & \Prof \arrow[ll, "u"]
		\end{tikzcd}
	\]
	where $u$ is the forgetful functor regarding $\bsP A$ as an object of $\Cat$; the unit of the relative adjunction is precisely the Yoneda embedding.

	Intuitively, the above definition already suggests a strategy to prove our main result, and in fact we're only one step far from its proof: the canonical equipment of $\caat$ with proarrows, i.e. the bicategory $\Prof$, can be presented as the Kleisli bicategory of the relative monad $\bsP_!$ (see \cite{fiore2016relative}): naive as it may seem, our main observation is that this is part of a general phenomenon: the Kleisli bicategory of \emph{every} yosegi box $\bsP : \cA \to \cK$ equips $\cA$ with proarrows.
\end{remark}
\begin{theorem}[Main theorem]\label{main-theorem}
	Let $j : \cA\subset \cK$ be the inclusion of a sub\hyp{}2\hyp{}category; the following conditions are equivalent:
	\begin{enumtag}{m}
		\item $\cK$ has a cocomplete Yoneda structure in the sense of \autoref{def-di-yoda}, \autoref{yoda-cocompleta}, with ideal of admissibles $\cA$, and such that $j_!\bsP$ and $\bsP^*\eta$ exist;
		\item the inclusion $j : \cA\subset \cK$ admits a Yoneda equipment $p \reladjL{j} u$;
		\item There exists a yosegi $\bsP :  \cA \subset \cK$ with relative unit $\myeta : j \To \bsP$.
	\end{enumtag}
\end{theorem}
We use a separate subsection to show each of the following implications:
\[\notag
	\begin{tikzcd}
		& \yose \arrow[Rightarrow,rd,"\text{§}\ref{yose-to-equ}"] &  \\
		\yone \arrow[Rightarrow,ru,"\text{§}\ref{yone-to-yose}"] &  & \equ \arrow[Rightarrow,ll,"\text{§}\ref{equ-to-yone}"]
	\end{tikzcd}
\]
In §\ref{yone-to-yose} we show that the presheaf construction of every cocomplete Yoneda structure defines a yosegi; somehow, the results in this section appear obvious in light of \autoref{sec:yosegi}, as our argument will simply be a formal analogue of \autoref{set-is-a-ianua}. The argument is quite elementary, but as far as we could find, this is an original contribution.

Subsequently, we show in §\ref{yose-to-equ} how the Kleisli bicategory of a yosegi $\bsP$ equips the domain of $\bsP$ with proarrows; again, the argument boils down to check that $\Kl(\bsP)$ satisfies the properties of \autoref{def-di-equi}. \autoref{yoda-equipment} is motivated by the fact that a proarrow equipment obtained as Kleisli bicategory of a yosegi contains enough information to rebuild the yosegi, or rather the Yoneda structure, it comes from. The result in §\ref{yose-to-equ} can be seen as an extension of \cite[4.2]{fiore2016relative} from the canonical Yoneda structure on $\Cat$ to a generic one on $\cK$.

Finally, we show in §\ref{equ-to-yone} how to recover a Yoneda structure from a Yoneda equipment; this is not a surprising result; again, the argument is quite long but completely elementary (it is essentially based on the intuition that equipments and cocomplete Yoneda structures yield equivalent formal category theories), and yet -to the best of our knowledge- it is an original contribution.
\subsection{A Yoneda structure defines a yosegi box}\label{yone-to-yose}
Given a Yoneda structure with admissibles $\cA$, we shall define a lax idempotent 2\hyp{}monad, relative to the inclusion $j : \cA\subset \cK$.

Keeping in mind the content of \autoref{set-is-ian:uno} and \autoref{set-is-ian:due}, the natural candidate is of course the presheaf construction $\bsP$ of the Yoneda structure itself. The correspondence $\bsP : A\mapsto \bsP A$ that sends $f : A\to B$ into $\bsP^*f = \Lan_{y_B\crc f}(y_A)$ defines a pseudofunctor since axioms \ref{ya:ter} and \ref{ya:qtr} hold. When the Yoneda structure is cocomplete, each $\bsP^* f$ has a left adjoint given by the left extension $\Lan_{y_A}(y_B\crc f)$ (see \autoref{nerve-real}). This remark makes the claim that each $f : A\to B$ is sent into an adjunction $\bsP_!f\dashv \bsP^*f$ appear straightforward.

We shall show that $\bsP_!$ defines a lax idempotent relative monad, with fully faithful unit $\myeta : j\To \bsP$, and moreover $\bsP^*\myeta$ and $j_!\bsP$ exist.

Let's start by defining the candidate unit and multiplication:
\begin{itemize}
	\item the Yoneda maps $y_A : A \to \bsP A$ of the Yoneda structure define the unit components $\myeta_A$ of the wannabe monad;
	\item we then define $\mymu_A := \bsP^* \myeta_A : \bsP\bsP A\to \bsP A$.
\end{itemize}
Let us observe that from this definition $\bsP_!\myeta_A \adjunct{\alpha}{\beta}\mymu_A$. Also, $\mymu_A$ is exactly the left extension $\Lan_{\myeta_{\bsP A}\myeta_A}(\myeta_A)$, whence its left adjoint corresponds to $\Lan_{\myeta_A}(\myeta_{\bsP A}\myeta_A)$. This will be a key remark (that's why we label the co/unit of $\bsP_!\myeta_A \dashv\mymu_A$ for future use): in fact, once we have shown that this data gives rise to a monad, it will automatically be lax idempotent.

Let's show that these data determine a monad then: first we show the commutativity of the unit diagrams
\[
	\begin{tikzcd}
		\bsP A \ar[r, "\myeta_{\bsP A}"]\ar[dr,equal]& \ar[d, "\mymu_A"] \bsP\jb\bsP A & \bsP A\ar[l, "\bsP_! \myeta_A"']\ar[dl,equal]\\
		& \bsP A & {}
	\end{tikzcd}
\]
The left triangle commutes, thanks to the chain of isomorphisms
\begin{align*}
	\mymu_A \crc \myeta_{\bsP A} & \cong \Lan_{\myeta_{\bsP A}\myeta_A}(\myeta_A)\crc \myeta_{\bsP A}        \\
	                             & \cong \Lan_{\myeta_{\bsP A}}(\Lan_{\myeta_A}\myeta_A)\crc \myeta_{\bsP A} \\
	                             & \cong \Lan_{\myeta_A}(\myeta_A) \cong 1_{\bsP A}.
\end{align*}
The right triangle commutes thanks to the chain of isomorphisms
\begin{align*}
	\mymu_A \crc \bsP_!(\myeta_A)       & \cong \Lan_{\myeta_{\bsP A}\myeta_A}(\myeta_A)\crc \Lan_{\myeta_A}(\myeta_{\bsP A}\myeta_A) \\
	\text{\autoref{KZ-cocomplete}} & \cong  \Lan_{\myeta_A}(\Lan_{\myeta_{\bsP A}\myeta_A}(\myeta_A)\crc\myeta_{\bsP A}\myeta_A) \\
	                               & \cong \Lan_{\myeta_A}(\myeta_A)  \cong 1_{\bsP A}.
\end{align*}
For what concerns the associativity property of $\mymu$, we have to show the commutativity of
\[
	\begin{tikzcd}
		\bsP\jb\bsP\jb\bsP A\ar[d, "\bsP_! \mymu_A"']  \ar[r,"\mymu_{\bsP A}"]& \bsP\jb\bsP A_s \ar[d, "\mymu_A"] \\
		\bsP\jb\bsP A \ar[r, "\mymu_A"'] & \bsP A
	\end{tikzcd}
\]
Now, the maps $\bsP_!\mymu$ and $\mymu_{\bsP A}$ share the same universal property, exhibiting both the left extension $\Lan_{\myeta_{\bsP\bsP A}\myeta_{\bsP A}}(\myeta_{\bsP A})$.
\begin{corollary}
	We found an adjunction $\mymu_{\bsP A}\dashv \bsP^*\mymu_A$ with invertible counit $\mymu_{\bsP A}\crc \bsP^*\mymu_A\To 1$; this entails that the right adjoint $\bsP^*\mymu_A$ is fully faithful.
\end{corollary}
\begin{corollary}
	The functor $\bsP$ is a ($j$-relative) lax idempotent 2\hyp{}monad.
\end{corollary}
\begin{proof}
	Recall from \autoref{app:contra} that $\bsP$ is lax idempotent if and only if $\bsP_!\myeta_A\dashv \mymu_A$; this is precisely a consequence of the definition.
\end{proof}
\subsection{A yosegi defines a proarrow equipment}\label{yose-to-equ}
\leavevmode
We shall show that the Kleisli bicategory of a yosegi is a proarrow equipment, and more precisely the free functor from $\cA$ to the bicategory of free $\bsP$-algebras equips $\cA$ with proarrows. Let us consider the diagram
\[
	\begin{tikzcd}
		& \cA \ar[d,phantom, "\overset{\myeta}\To"]\arrow[rd, "p"] \arrow[ld, "j"'] &  \\
		\cK  & {} & \Kl(\bsP) \ar[ll, "u"]
	\end{tikzcd}
\]
where the relative adjunction $p \reladjL{j} u$ has a relative unit $\myeta : j\To up$. Our claim is that the free functor $p : \cA \to \Kl(\bsP)$ satisfies axioms \ref{pe:due}--\ref{pe:ter} yielding a proarrow equipment; it will be a Yoneda equipment, as a direct consequence of the lax idempotency of $\bsP$ (by construction, the composition $up$ coincides with $\bsP_!$). Of course, the free functor can be thought as the identity on objects: $A\in\cA$ has an associated $\bsP$-algebra
\begin{itemize}
	\item $p$ is a locally fully faithful 2\hyp{}functor: indeed, $p$ acts as $\bsP_!$ on 0\hyp{} and 1\hyp{}cells, so we have isomorphisms
	      \[
		      \begin{array}{c}
			      \Lan_{\myeta_A}(\myeta_B \crc f)\longrightarrow \Lan_{\myeta_A}(\myeta_B \crc g) \\ \hline
			      f \longrightarrow g
		      \end{array}
	      \]
	      as a consequence of $\Lan_{\myeta_A}$ and $y_B$ being fully faithful.
	\item Every $p(f)$ has a right adjoint in $\Kl(\bsP)$: this is evident, since $\bsP$ is a yosegi.
\end{itemize}
\subsection{A Yoneda equipment defines a Yoneda structure}\label{equ-to-yone}
\leavevmode
Quite understandably, this is the most difficult implication.

Let us first recall what we have: a triangle
\[
	\begin{tikzcd}
		& \cA \ar[d,phantom, "\overset{\myeta}\To"]\arrow[rd, "p"] \arrow[ld, "j"'] &  \\
		\cK  & {} & \cM \ar[ll, "u"]
	\end{tikzcd}
\]
such that
\begin{enumtag}{eq}
	\item \label{eq:uno} $p$ is a proarrow equipment in the sense of \autoref{def-di-equi};
	\item \label{eq:due} $\myeta$ is the fully faithful relative unit of a relative adjunction $p \reladjL{j} u$;
	\item \label{eq:ter} the composition $\bsP_! = up$ is a $j$-relative lax idempotent 2\hyp{}monad.
\end{enumtag}
\begin{remark}\label{ideal-of-admissibles}
	First of all we have to define the admissible arrows; we say $f \in \cK$ is admissible if the arrow $\bsP^*(f)$ exists (as always, we adopt the convention that if $B\notin \cA$, then $\bsP B = j_!\bsP B$) and has the property indicated in \autoref{nervous}. Obviously, this class of 1\hyp{}cells is closed under composition, but there's no hope it absorbs on the right. \cite[Thm. 18]{walker} faces the same issue, which indicates it is likely impossible to provide better than a composition\hyp{}closed class of 1\hyp{}cells, instead of a full ideal; since the nervous maps of \autoref{nervous} can play the r\^ole of such composition\hyp{}closed class, we call them as a whole the \emph{nerve} of the Yoneda equipment.
\end{remark}
We will show that $\cK$ has a Yoneda structure having
\begin{itemize}
	\item Yoneda embedding the relative units of \ref{eq:due};
	\item nerve triangles given by
	      \[\label{bieffeuno}
		      \begin{tikzcd}
			      & A \arrow[ld,bend right=20, "f"'] \arrow[rd,bend left=20, "\myeta_A"] &  \\
			      B \arrow[r, "\myeta_B"] & \bsP B \arrow[r, "\bsP^*f"] & \bsP A
		      \end{tikzcd}\]
	      (thus the 1\hyp{}cell $B(f,1)$ is defined as $\bsP^*(f)\crc \myeta_B$);
	\item filled by 2\hyp{}cells $\chi^f$, defined as the pasting
	      \[
		      \label{lecelle}
		      \begin{tikzcd}
			      A \arrow[d, "f"'] \arrow[r, "\myeta_A"] \arrow[rd,phantom, "\scriptstyle\gamma\Swarrow"] & \bsP A \arrow[d, "\bsP_!f"] \arrow[rd,equal, bend left] \arrow[rd, "\Swarrow",phantom] &  \\
			      B \arrow[r, "\myeta_B"'] & \bsP B \arrow[r, "\bsP^*f"'] & \bsP A
		      \end{tikzcd}
	      \]
	      obtained from the unit of $\bsP_!f\dashv\bsP^*f$ and the canonical 2\hyp{}cell $\gamma$.
\end{itemize}
We shall show that these data satisfy the axioms of Yoneda structure as stated in \autoref{def-di-yoda}, \ref{ya:uno}--\ref{ya:qtr}.
\begin{enumtag}{ya}
	\item The triangles in \eqref{lecelle} are left extensions: since $\bsP_! f \cong \Lan_{\myeta_A}(\myeta_B f)$, the left adjoint $\bsP^* f$ can be characterized as $\Lan_{\myeta_B f}(\myeta_A)$. But now this entails that
	\begin{align*}
		B(f,1) & \cong \bsP^* f\crc \myeta_B                          \\
		       & \cong \Lan_{\myeta_B}(\Lan_f(\myeta_A))\crc \myeta_B \\
		       & \cong \Lan_f \myeta_A.
	\end{align*}
	\item The same triangles in \eqref{lecelle} exhibits $(f,\chi^f)$ as $\LIFT_{B(f,1)}\myeta_A$; since $B(f,1)$ is defined as the composition $\bsP^*(f)\crc \myeta_B$, and the left lifting $\Lift_{uv}(w)$ along $u\crc v$ unload as $\Lift_v(\Lift_uw)$, we shall show that
	\[\Lift_{\myeta_B}(\Lift_{\bsP^* f}\myeta_A)\cong f.\]
	In order to check this, it suffices to show that $\Lift_{\bsP^* f}\myeta_A\cong \myeta_B \crc f$; whereas this is true indeed, the fully faithfulness of $\myeta_B$ will imply the result.

	But now, we have $\Lift_{\bsP^* f}(\firstblank)\cong \bsP_!f\crc \firstblank$, and this already concludes:\footnote{It is indeed a general fact that an adjunction $v\dashv u$ induces an adjunction $\Lift_u \dashv v\crc\firstblank$ by uniqueness of adjoints (a similar result suitably dualized, holds for right liftings and extension, where maybe this is more commonly encountered in practice).} indeed
	\begin{align*}
		\bsP_!f\crc \myeta_A & \cong \Lan_{\myeta_A}(\myeta_B f)\crc \myeta_A \\
		                     & \cong \myeta_B \crc f
	\end{align*}
	Now, observe that $\Lift_{\myeta_B}(\Lift_{\bsP^* f}\myeta_A)\cong \Lift_{\myeta_B}(\myeta_B f)$ and that $\myeta_B\crc\firstblank$ is fully faithful and a right adjoint for  $\Lift_{\myeta_B}(\firstblank)$; this implies the existence of an isomorphism $\Lift_{\myeta_B}(\myeta_B f)\cong f$. A similar argument shows that the lifting is absolute, i.e. that $fg\cong \Lift_{\bsP^*f \crc \myeta_B}(\myeta_A\crc g)$ for every $g : X\to A$.
	\item the map $\myeta_A$ is dense, namely $\Lan_{\myeta_A}(\myeta_A)$ exists, is pointwise and isomorphic to the identity $ 1_{\bsP A}$; this follows from functoriality of $\bsP$ by a simple `double counting' argument, evaluating $\bsP^*(1_A)$ in two ways: since $\bsP^*$ is a functor, $\bsP^*(1_A)\cong 1_{\bsP A}$; on the other hand, since $\bsP^*f\cong\Lan_{\myeta_B f}(\myeta_A)$ it follows that $\bsP^*(1_A)\cong \Lan_{\myeta_A}(\myeta_A)$.
	\item There is an isomorphism
	\[\Lan_{gf}(\myeta_A)\cong \Lan_{\myeta_Bf}(\myeta_A)\crc\Lan_g(\myeta_B);\]
	this again follows from functoriality of $\bsP$, together with lax idempotency: the left hand side of the isomorphism equals $\bsP^*(gf)\crc \myeta_A$; on the other side there is $\bsP^* f\crc \bsP^*g\crc \myeta_A$. The two sides are visibly isomorphic, and this concludes the proof.\qed
\end{enumtag}
% !TEX root = ../yosegi.tex
\section{The importance of being $\bsP 1$}
\label{sec:importance}
% \epigraph{One of the more mysterious bits of structure possessed by the 2\hyp{}category $\Cat$ is its duality
% involution.}{\cite[§1]{shulmano2016contra}}
\subsection{Representable yosegi}
\label{sec:repre}
Certainly, the presheaf construction of $\Cat$ satisfies more properties than a generic cocomplete Yoneda structure: here we study some of these properties, and in particular the fact that $\bsP$ is a `semi-representable' functor of the form $\psh[\Omega]{\firstblank}$, a fact that admits an abstraction in \autoref{Unicity}. Note how such a characterization of semi\hyp{}representable yosegi borrows a lot both from the definition of \emph{fibrational cosmos} (where $\bsP$ has a left pseudo\hyp{}adjoint $\bsP^\dag$, the functor we call $\bsQ$ in \autoref{powering-relative}), and the theory of Weber's 2\hyp{}toposes (in \cite{weber2007yoneda} it is shown that $\Set_* \to \Set$ is a classifying fibration). Our axioms sit in the middle between these two poles: the aim of a future work \cite{} will be to clarify which properties of $\Omega$ turn the functor $\psh[\Omega]{\firstblank}$ into a yosegi, endowing $\cK$ with a Yoneda structure even if not with a 2\hyp{}topos structure.

More in detail, the scope of the present section is to show how under very mild assumptions on $\bsP$ and on the embedding $\cA\subset \cK$, the presheaf construction is `determined by its value on $1$', in the sense that there is a natural isomorphism $\bsP A \cong [A^\dual, \bsP 1]$, for a suitable restricted internal hom $[\firstblank, \secondblank] : \cA^\op\times \cK \to \cK$. We freely employ the notion of a duality involution, introduced in \cite{weber2007yoneda} in relation to discrete opfibrations, and formalized in \cite{shulmano2016contra}.
\begin{notation}\label{blanket}
	We make the following assumptions while fixing a notation that will accompany us all along the section:
	\begin{enumtag}{r}
		\item \label{r:uno} $j : \cA\subset \cK$ is the inclusion of a full sub\hyp{}2\hyp{}category $\cA$ of a 2\hyp{}category $\cK$
		\item \label{r:due} $\cK$ has finite limits and a duality involution $(\firstblank)^\dual : \cK^\co \to \cK$ (see \cite{shulmano2016contra} for a definition);
		\item \label{r:tre} finite limits of objects in $\cA$ still lie in $\cA$, and $A^\dual \in \cA$;
		\item \label{r:qtr} $\cA$ is enriched over $\cK$ with a functor $\llbracket \firstblank,\secondblank\rrbracket : \cA^\op\times\cA \to \cK$;
		\item \label{r:cin} for all $A\in \cA$ there is a functor $[A, \firstblank ]: \cK \to \cK$ and an isomorphisms
		\[
			[A \times B, C] \cong [A, [B, C]]			
			% \begin{tikzcd}
			% 	& \mathcal A \arrow[rd, "\firstblank\times A"] \arrow[ld, "j"] &  \\
			% 	\mathcal K \arrow[rr, "{[A,\firstblank]}"'] &  & \mathcal K
			% \end{tikzcd}
		\]
		% making $\firstblank\times A$ the $j$-relative left adjoint of $[A,\firstblank]$;
		% \todo[inline]{Qui c'è un problema: se $\cA$ è chiusa per limiti finiti, $\firstblank\times A$, quando ha dominio $\cA$ deve cadere in $\cA$, e quando però ha dominio $\cK$, non può essere aggiunto $j$-relativo

		% Hai ragione e non so rispondere. So scrivere l'equazione che serve essere vera, essa è $[j(A \times B), C] \cong [jA, [jB, C]]$.}
		\item when evaluated on both objects in $\cA$, there is a natural isomorphism
		\[
			\llbracket A,A'\rrbracket \cong [A,A'].
		\]
	\end{enumtag}
\end{notation}
\begin{remark}
	As a consequence of \ref{r:due} $\cK$ has finite products (and a terminal object).
	In case $\cK$ is cartesian closed, the assumptions in \autoref{blanket} are obviously satisfied: if  $\cK$ is powered over $\cA$, we are asking that the powering and the enrichment coincide.

	Somehow, this setting axiomatizes the properties of $j : \caat\subset\Cat$; its codomain is not a cartesian closed category, but for every $A\in\caat$ and $X\in \Cat$ there is a legitimate category of functors $[A,X]\in\Cat$.
\end{remark}
\begin{definition}[Powering relative left adjoint]\label{powering-relative}
	Under the assumptions in \autoref{blanket}, given a yosegi $\bsP : \cA \to \cK$, its contravariant part $\bsP^*$ defines a functor $\cA^\coop \to \cK$. A \emph{powering relative left adjoint} is now a functor $\bsQ: \cA^\coop \to \cK $ with the property that $[A, \bsQ A']^\dual \cong [A', \bsP^* A]$, naturally in $A,A'\in\cA$.
\end{definition}
\begin{lemma}\label{Unicity}
	In the above assumptions, given a yosegi $\bsP : \cA \to \cK$, the following conditions are equivalent:
	\begin{enumtag}{p}
		\item \label{p:uno} $\bsP^*$ has a powering-relative left adjoint $\bsQ$.
		\item \label{p:due} $\bsP^*$ is \emph{semi\hyp{}representable} by $\bsP^*(1)$, i.e. $\bsP^* A\cong [A^\dual, \bsP^* 1]$.
		\item \label{p:tre}  the functor $\bsP^*(\firstblank^\dual)^\dual$ is the powering relative left adjoint $\bsQ$;

	\end{enumtag}
\end{lemma}
\begin{proof} To prove that \ref{p:uno} implies \ref{p:due}, we first observe that $\bsQ(1)^\dual \cong \bsP^*(1)$, in fact
	\begin{align*}
		\bsQ(1)^\dual \cong [1, \bsQ(1)]^\dual  \cong [1, \bsP^*(1)] \cong  \bsP^*(1).
	\end{align*}
	Now it is enough to follow the chain of isomorphisms
	\begin{align*}
		\bsP^* A & \cong [1,\bsP^* A] \cong [A, \bsQ 1]^\dual                    \\
		         & \cong [A^\dual,  \bsQ( 1)^\dual] \cong [A^\dual,  \bsP^*(1)].
	\end{align*}
	To prove that \ref{p:due} implies \ref{p:tre}, we follow the chain of isomorphisms
	\begin{align*}
		[B,\bsP^* A] & \cong [B, [A^\dual, \bsP^* 1]]       \\
		             & \cong [A^\dual , [B, \bsP^* 1]]      \\
		             & \cong [A, [B, \bsP^* 1]^\dual]^\dual \\
		             & \cong [A, \bsP^*(B^\dual)^\dual].
	\end{align*}
	Finally, \ref{p:tre} trivially implies \ref{p:uno}.
\end{proof}
\begin{remark}
	If any of these two conditions is satisfied, $\bsQ$ and $\bsP$ determine each other under the natural isomorphisms $\bsP(A^\dual)^\dual \cong \bsQ A$ and $\bsQ(A^\dual)^\dual\cong \bsP A$. As a consequence, the functor $\bsQ$ enjoys similar properties:
	\begin{enumtag}{q}
		\item \label{psharp:uno} $\bsQ$ has a powering relative right adjoint;
		\item  \label{psharp:due} $\bsQ$ is \emph{semi\hyp{}representable} by $\bsQ(1) = \bsP(1)^\dual$, i.e. $\bsQ A\cong [A^\dual, \bsQ 1]$.
		\item \label{psharp:tre}  the functor $\bsQ(\firstblank^\dual)^\dual$ is the powering relative right adjoint.
	\end{enumtag}
\end{remark}
\begin{remark}\label{who-s-Q}
	In the case of the canonical Yoneda structure on $\Cat$, $\bsQ$ is the `covariant presheaf construction' $A\mapsto [A,\Set]^\op$, and the covariant Yoneda embedding is obtained from $z_A^\op : A^\op\to [A,\Set]$; the observation that $\bsQ\dashv \bsP$ dates back to Street  and partly motivates the definition of fibrational cosmos (see \cite{StreetCosmoi1974}, but more \cite{street1980cosmoiof}).
\end{remark}
In the following remarks we re\hyp{}enact the essential features of the bicategory of profunctors, showing that the Kleisli bicategory of a representable yosegi enjoys many of the properties of the bicategory $\Prof$.
\begin{proposition}\label{mating-season}
	Let $f : A\leftrightarrows A' : u$ be an adjunction in $\cA$, and let us employ the shorthand $N_f := A'(f,1) : A'\to \bsP A, M_u := A'(1,f) : A\to \bsQ A'$ to denote the extensions in \eqref{triangle-data} and \eqref{right-triangle-data}; then the 1\hyp{}cells $N_f, M_u$ correspond each other under the adjunction isomorphism $\cK(A', \bsP A)\cong \cK(A, \bsQ A')$.
\end{proposition}
\begin{proof}
	Unit and counit of the adjunction $\bsQ\adjunct{\eta}{\epsilon}\bsP$ are the maps
	\begin{align*}
		\eta : 1 \To \bsP\bsQ,        & \quad \eta_{A,\bsP} \in \cK\Big(A, \big[\big([A,\Omega]^\op\big)^\op,\Omega\big]\Big) = \cK\Big(A,\big[[A,\Omega],\Omega\big]\Big)           \\
		\varepsilon : \bsQ\bsP \To 1, & \quad \epsilon_{A,\bsP} \in \cK^\coop\Big(\big[[A^\op,\Omega],\Omega\big]^\op,A\Big) = \cK\Big( A^\op, \big[[A^\op,\Omega],\Omega\big] \Big)
	\end{align*}
	determined as \emph{mates} of evaluations counits in $\cK$. If we explicitly write down the bijection linking $\cK(A, \bsP A')\cong \cK(A' \bsP A)^\op$ in function of unit and counit $\eta,\epsilon$ the claim of \autoref{mating-season} boils down to the fact that
	\begin{align*}
		M_u & = \Big(A \xto{\eta_{A,\bsP}} \bsQ \bsP A \xto{\bsQ (N_f)} \bsQ A'\Big)       \\
		N_f & = \Big(A' \xto{\varepsilon_{A',\bsP}} \bsP\bsQ A'\xto{\bsP (M_u)}\bsP A\Big)
	\end{align*}
	In order to see that this is true, it suffices to notice that the first composition shares the same universal property of $M_u = \Lan_u z_{A'}$, and the second one does the same for $N_f = \Lan_f y_A$.
\end{proof}
\begin{remark}\label{2-sided-yosegi}
	Let $\bsP$ be semi\hyp{}representable, and let $\bsQ$ denote its left pseudo\hyp{}adjoint as in \autoref{Unicity}; given $f :A\to A'$, there is an isomorphism of categories $\Kl(\bsP)\cong \Kl(\bsQ)$, such that $B(f,1)\dashv B(1,f)$ (letting the identification implicit).
\end{remark}
\begin{remark}[$*$-autonomy of equipments]
	\label{compact-equip}
	In the present remark we want to prove that the equipment generated by a representable yosegi $\bsP$ is a $*$-autonomous category: this is a formal analogue of the fact \cite[§5]{cattani2005profunctors} that the bicategory of profunctors is $*$-autonomous because it is compact closed. In order to reach such a result we have to prove that the Kleisli bicategory $\Kl(\bsP)$
	\begin{itemize}
		\item is symmetric monoidal;
		\item has an involution $(\firstblank)^\star$ such that for each three objects there is an isomorphism
		      \[
			      \cE(A\otimes B, C^\star)\cong \cE(A, (B\otimes C)^\star)
		      \]
	\end{itemize}
	Of course, generic Kleisli bicategories do not have finite limits; in the present case representability of $\bsP$ gives right away that the coproduct $A\coprod_{\cA}B$ has the universal property of $A \times_{\kern-.2em\Kl} B$; indeed, exploiting the description of the objects of $\Kl(\bsP)$ as free $\bsP$-algebras, we have that $\bsP(A\coprod_{\cA} B)\cong \bsP A\times \bsP B$. %, and the projections are algebra morphisms
	It now remains to verify the second condition, of course with respect to the duality involution $(\firstblank)^\dual$: in order to do it, it suffices to recall how 1\hyp{}cells in $\Kl(\bsP)$ are defined: we have the following two chains of isomorphisms.
	\begin{align*}
		\Kl(A\times B, C^\dual)   & \cong \cK(C^\dual, \bsP(A\times B))                         \\
		                          & \cong \cK(C^\dual, [A^\dual, \bsP B])                       \\
		                          & \cong \cK(C^\dual \times A^\dual, \bsP B)                   \\
		                          & \cong \cK(C^\dual \times A^\dual\times B^\dual, \bsP 1)     \\
		\Kl(A, (B\times C)^\dual) & \cong \Kl(A, B^\dual\times C^\dual)                         \\
		                          & \cong \cK(B^\dual\times C^\dual, \bsP A)                    \\
		                          & \cong \cK(B^\dual\times C^\dual\times A^\dual, \bsP 1).\qed
	\end{align*}
\end{remark}
After having studied a particular class of equipments generated by representable yosegi, it is natural to close the circle investigating Yoneda structures induced by representable yosegi. As a consequence of \autoref{Unicity}, the couple $\bsQ, \bsP$ is a powering relative adjunction precisely when one of them is semi-representable; in this sense, studying the representability of $\bsP, \bsQ$ boils down to study the induced adjunction $\bsQ\dashv \bsP$. This leads to the notion of \emph{ambidextrous} Yoneda structure.
\subsubsection{Ambidextrous Yoneda structures}\label{two-side-ys}
Intuitively, an ambidextrous Yoneda structure arises when a Yoneda structure on a 2\hyp{}category $\cK$ is promoted to be the `left part' of a pair of Yoneda structures; the left part defined in \autoref{def-di-yoda}, and the right part, satisfying completely dual axioms that we now outline.

Quite surprisingly, neither the very mild assumptions under which such a structure exists in $\cK$, nor the very natural motivation that the canonical Yoneda structure of $\Cat$ is ambidextrous, ignited a thorough study of such systems. The authors concocted the notion of ambidexterity in a series of private conversations with D. Bernardini, and immensely profited from the draft of a paper systematically studying Yoneda structures \cite{bernardini2017yoneda}.
\begin{definition}[Right Yoneda structure]\label{rightyoneda}
	A \emph{right Yoneda structure} on a 2\hyp{}category $\cK$ consists of \emph{Yoneda data}:
	\begin{enumtag}{yd}
		\item \label{ryd:uno} An ideal $\cJ_{\bsQ}$ of 1\hyp{}cells that we call `right admissible'; arrows in this ideal determine right admissible objects in the class $\cJ_{\bsQ, 0}\subseteq \cK$: $A$ is right admissible if such is its identity 1\hyp{}cell.
		\item \label{ryd:due} Each right admissible object $A$ has a `Yoneda arrow' $z_A : A \to \bsQ A$;
		\item \label{ryd:tre} every right admissible morphism $f : A\to B$ with right admissible domain fits into a triangle
		\[\label{right-triangle-data}
			\begin{tikzcd}
				{} & A\ar[d,phantom,"\sigma_f\Searrow"]\ar[dr, "f"]\ar[dl, "z_A"'] & {} \\
				\bsQ A & {} & \ar[ll, "{B(1,f)}"] B
			\end{tikzcd}
		\]
		filled by a 2\hyp{}cell $\sigma_f : z_A \to B(1,f)\crc f$.
	\end{enumtag}
	These data are	organized to enjoy the following properties:
	\begin{enumtag}{ya}
		\item \label{rya:uno} The pair $\langle B(1,f), \sigma_f\rangle$ exhibits the right extension $\Ran_fz_A$.
		\item \label{rya:due} The pair $\langle f, \sigma_f\rangle$ exhibits the absolute right lifting $\RIFT_{B(1,f)}z_A$.
		\item \label{rya:ter} The pair $\langle 1_{\bsQ A}, 1_{z_A} \rangle$ exhibits  the pointwise right extension $\Ran_{z_A}z_A$.
		\item \label{rya:qtr} Given a pair of composable 1\hyp{}cells $A \xto{f} B\xto{g} C$, the
		pasting of 2\hyp{}cells
		\[
			\begin{tikzcd}[column sep=large, row sep=large]
				A\ar[d, "f"']\ar[rr, "z_A"{name=yonA}] & & \bsQ A\\
				B \ar[r, "z_B"{name=yonB}]\ar[d, "g"'] & \bsQ B\ar[ur, "\bsQ f"']\\
				C\ar[ur, "{C(g,1)}"'] \ar[from=yonA, to=yonB, shorten >=2mm, shorten <=4mm, Rightarrow, "\sigma_{z_B f}"] \ar[from=yonB, shorten >=4mm, shorten <=4mm, Rightarrow, "\sigma_g"]
			\end{tikzcd}
		\]
		exhibits the  extension $\Ran_{gf}y_A = C(1,g\crc f)$.
	\end{enumtag}
\end{definition}
\begin{remark}[A few remarks on the dual axioms]
	Axioms \ref{rya:due} and \ref{rya:ter} appear to have been incorrectly dualized: the contravariant Yoneda embedding $z_A$ is still dense, so why is its \emph{right} extension the identity in \ref{rya:ter}? We shall expect from \ref{rya:due} to assert that $B(1,f) \dashv f$ (relative to $z_A$), so why is the axiom about a left lifting instead? It turns out that the contravariance of $z_A$ plays a r\^ole: indeed, \eqref{right-triangle-data} results from the application of the duality involution to the triangle
	\[\notag
		\label{nonopped-triangle-data}
		\begin{tikzcd}
			{} & A\ar[d,phantom,"\sigma_f\Nwarrow"]\ar[dr, "f"]\ar[dl, "z_A"'] & {} \\
			\bsQ A & {} & \ar[ll, "{B(1,f)}"] B
		\end{tikzcd}
	\]
	and similarly, in axiom \ref{rya:ter} what we see is the image of the isomorphism $\Lan_{z_A}z_A\cong 1_{\bsQ A}$ under $(\firstblank)^\dual$. Thus, we can safely assert that \ref{rya:due} is asking that $B(1,f)\dashv f$ forms a $z_A$-relative adjunction.
\end{remark}
There is no way to ensure, in general, that a left and a right Yoneda structure live together in the same category; yet, this is the most common case (it happens, for example, for the canonical Yoneda structure on $\VCAT$); in fact, each of these 2\hyp{}category is naturally equipped with a duality involution and the Yoneda structure is representable by $\bsP 1 = \cV$; the most natural request is then that the powering relative left adjunction $\bsQ\dashv \bsP$ generates a pair of nicely interacting Yoneda structures.
\begin{definition}[Ambidextrous Yoneda structure]\label{ambidex}
	Let $\cK$ be a 2\hyp{}category.	An \emph{ambidextrous} Yoneda structure consists of a left Yoneda structure (\autoref{def-di-yoda}) and a right Yoneda structure (\autoref{rightyoneda}) on $\cK$ such that
	\begin{itemize}
		\item $\bsQ, \bsP$ share the same admissible (this allows for the existence of both $y_A$ and $z_A$);

		\item The two Yoneda embeddings of $\bsP$ and $\bsQ$ are linked by the relation $z_A = y_{A^\dual}^\dual$; this allows for the `interdefinability' of the left and right extensions using the fact that $\ran_f(g) \cong (\lan_{f^\dual}(g^\dual))^\dual$.
	\end{itemize}
\end{definition}
\subsubsection{A glance at Isbell duality}
In basic category theory, Isbell duality states that there is an adjunction between the category $\psh{A}$ of presheaves and the category $[A,\Set]$ of co\hyp{}presheaves on $A$; somehow, this can be figured as a non\hyp{}additive, multi\hyp{}object version of the  correspondence between the category of left $R$-modules and the category of right $R^\op$-modules, given a ring $R$.

Given a small category $A$, Isbell duality is an adjunction $\cO : \psh{A}\leftrightarrows [A,\Set]^\op : \Spec$ defined by the functors
\begin{gather}
	\cO_A : F \mapsto \llambda a.	\psh{A}(F, y_A(a))\\
	\Spec_A : G \mapsto \llambda a.[A,\Set]^\op(G,z_A(a))
\end{gather}
where $z_A$ is the contravariant Yoneda embedding of \autoref{who-s-Q}. This very definition shows how tightly Isbell duality is linked to the pair of Yoneda embeddings $y_A,z_A$ of a small category into its co\fshyp{}presheaf categories $A$; the aim of the present paragraph is to show how many properties of Isbell duality are to a large extent completely formal.
\begin{lemma}\label{formal-isbe}
	Let $A$ be a small category; the functors $\cO_A\dashv\Spec_A$ of Isbell duality are such that $\cO_A(F)\cong \Lan_{y_A}z_A(F)$ and $\Spec_A(G)\cong\Lan_{z_A}y_A(G)$.

	In other words,	Isbell duality arises from the nerve\hyp{}realization of \autoref{nerve-real} associated to the two Yoneda embeddings $y_A, z_A$.
\end{lemma}
\begin{proof}
	An immediate application of the coend formula for a left Kan extension shows that $\Lan_{y_A}{z_A}(F)\cong \llambda a.	\psh{A}(F, y_A(a))$; now, uniqueness of a right adjoint plus \autoref{nerve-real} entail that there is an isomorphism $\Lan_{z_A}y_A(G)\cong \llambda a.[A,\Set]^\op(G,z_A(a))$.
\end{proof}
This motivates the following definition.
\begin{definition}\label{isbellone}
	Let $\bsQ\dashv\bsP$ be an adjunction associated to the representable yosegi $\bsP$; if $A$ is an object of $\cA\subset \cK$, there is an adjunction $\cO_A\dashv \Spec_A$ obtained as
	\[
		\begin{tikzcd}
			{} & A\ar[dl, "y_A"']\ar[dr, "z_A"] & {} \\
			\bsP A \ar[rr, shift left=1mm, "\cO_A"]& {} & \bsQ A \ar[ll, shift left=1mm, "\Spec_A"]
		\end{tikzcd}
	\]
	taking the left extensions $\cO_A = \Lan_{y_A}(z_A)$ and $\Spec_A = \Lan_{z_A}(y_A)$.
\end{definition}
We say that a presheaf $F\in\psh{A}$ is \emph{Isbell self\hyp{}dual} if it is fixed by the monad $\Spec_A\crc\cO_A$ or (equivalently) if it is fixed by the comonad $\cO_A\crc\Spec_A$. Fully faithfulness of the Yoneda embeddings entail that all representable co/presheaves are Isbell self\hyp{}dual; this has an evident formal analogue in light of \autoref{formal-isbe}.
\subsection{Formal enrichment and profunctors}\label{sec:formal}
The following long remark shows how in a cocomplete Yoneda structure, or better, in its associated yosegi, there is a `calculus of profunctors'; somehow, this is to be expected: it can be seen as an explicit presentation of the Yoneda equipment generated by the yosegi, as the bicategory of profunctors $\Prof_{\bsP}(\cA)$ coincides with the Kleisli bicategory of $\bsP$, regarded as a relative monad, in the sense of \autoref{kleibi}.
\begin{remark}\label{for-enrich}
	In a representable yosegi $\bsP : \cA \to \cK$, units are maps in $\cK$ of the form $A \to [A^\dual, \bsP 1]$; the product $A^\dual\times A$ is admissible in the Yoneda structure generated by $\bsP$, so we can consider the admissible maps
	\[
		\fh : A^\dual\times A\to \bsP 1.
	\]
	These maps play the r\^ole of \emph{internal homs}, so that admissible objects are -not exactly, but in a clearly evocative way- $\bsP1$\emph{-enriched}: for every small category $A$ composition maps are given by a family of functions
	\[
		c_{abc} : A(a,b)\times A(b,c)\to A(a,c)
	\]
	such that $c_{a,\firstblank,c}$ is a cowedge, and in fact an initial one.
\end{remark}
A `semi-classical' point of view on the situation is the following: to characterize $\fh_A$ as a certain coend, we shall write it as a certain weighted colimit (because co/ends are precisely hom\hyp{}weighted co/limits).

By definition of what is a weighted colimit in a Yoneda structure \cite[§4]{street1978yoneda}, the fact that $\int^{x}A(a,x)\times A(x,a')\cong A(a,a')$ means that the left lifting of $y_{A^\op\times A}$ along $\bsP 1(\fh_A,1)$, is $\fh_A$, and such lifting is absolute; this yields that $\fh_A \reladjL{y_{A^\op\times A}} \bsP 1(\fh_A,1)$; this final request means that the triangle
\[
	\begin{tikzcd}
		& A^\op\times A \ar[d,phantom, "\overset{\chi^{\fh_A}}\To", description]\arrow[ld, "y_{A^\op\times A}"'] \arrow[rd, "\fh_A"] &  \\
		\bsP(A^\op\times A) & {} & \bsP 1 \arrow[ll, "{\bsP 1(\fh_A,1)}"]
	\end{tikzcd}
\]
is an absolute left lifting; this is precisely axiom \ref{ya:due} applied to $\fh_A$.

It is worth to investigate how far this reconstruction procedure can go, in order to recover a full-fledged coend calculus: in view of this, $\fh_A$ will play the same r\^ole of the hom functor $\hom_A$ and thus the above isomorphism, expressing $\fh_A$ as `the coend' of $\fh_A \times \fh_A$, can be read as the isomorphism $\hom_A\diamond \hom_A\cong \hom_A$ witnessing that $\hom_A$ is the identity profunctor $A \overset{\hom_A}\pto A$.
\begin{remark}
	The 1\hyp{}cell $\bsP 1(\fh_A,1)$ admits a left adjoint $\Lan_{y_{A^\dual\times A}} \fh_A$; computing this Kan extension in $\Cat$ we get
	\begin{align*}
		\Lan_{y_{A^\dual\times A}} \fh_A(F) & \cong \int^{A^\dual\times A} [A^\dual\times A,\Set](y(a,a'), F)\times \fh_A(a,a') \\
		                                    & \cong \int^{A^\dual\times A} F(a,a')\times A(a,a')                                \\
		                                    & \cong \int^A F(a,a).
	\end{align*}
	This entails that the adjunction $\Lan_{y_{A^\dual\times A}} \fh_A\dashv \bsP 1(\fh_A,1)$ formalizes the adjunction $\int^A : [A^\dual\times A, \Set] \leftrightarrows \Set : \hom\pitchfork\firstblank$, where the left adjoint $\int^A$ takes the coend of a functor $F : A^\op\times A\to \Set$, and the right adjoint is $\hom\pitchfork X : (a,a')\mapsto X^{\hom(a,a')}$ (this is a general fact: taking the colimit weighted by $W$ is left adjoint to powering with the weight).

	It is now straightforward to go further: axiom \ref{ya:due} entails that in a Yoneda structure an admissible object $A$ for which $A\times A^\op$ is still admissible (this translates, for the corresponding yosegi, into $\cA$ being closed under product and duality involution), there are absolute left liftings
	\[
		\begin{tikzcd}
			& A^\op\times A \ar[d,phantom, "\To", description]\arrow[ld, "y_{A^\op\times A}"'] \arrow[rd, "\fh_A"] &  \\
			\bsP(A^\op\times A) & {} & \bsP 1 \arrow[ll, "{\bsP 1(\fh_A,1)}"]
		\end{tikzcd}
	\]
	in which $\bsP 1(\fh_A, 1)$ has a left adjoint, precisely $\Lan_{y_{A^\dual\times A}}\fh_A$. Such a left adjoint is the functor that `coends' an endo\hyp{}profunctor $F : A^\lor\times A\to \bsP 1$.
\end{remark}
\begin{definition}[The bicategory of $\bsP$-profunctors]
	Let $\cK$ be a 2\hyp{}category, and $\bsP : \cA\to \cK$ a yosegi on $\cK$; then, we define the bicategory $\Prof_{\bsP}(\cA)$ of $\bsP$-profunctors as follows:
	\begin{itemize}
		\item 0\hyp{}cells are those in $\cA$;
		\item 1\hyp{}cells $f : A \pto B$ are the 1\hyp{}cells $f : B\to \bsP A$ in $\cK$; when $\bsP$ is representable, the 1\hyp{}cells of $\Prof_{\bsP}(\cA)$ correspond to 1\hyp{}cells $B\times A^\dual \to \bsP 1$ by mating;
		\item 2\hyp{}cells $\alpha : f\To g$ are all those in $\cK$.
	\end{itemize}
	The composition law $A\overset{f}\pto B\overset{g}\pto C$ is defined, reminding profunctor composition, as follows: given $f : B\to \bsP A$ and $g : C\to \bsP B$ we define the composition $g \bullet_{\bsP} f$ using the diagram
	\[
		\begin{tikzcd}
			{} & B\ar[r, "f"]\ar[d, "y_B"'] & \bsP A\\
			C \ar[r, "g"']& \bsP B \ar[ur, "\bsP_!f"']& {}
		\end{tikzcd}
	\]
	as $\bsP_! f\crc g$. Clearly, axiom \ref{ya:ter} of the Yoneda structure generated by $\bsP$ entails that $y_A : A\to \bsP A$ is the identity of $A\in \Prof_{\bsP}(\cA)$; for what concerns associativity, we have to show that $(h\bullet_{\bsP} g)\bullet_{\bsP} f \cong h\bullet_{\bsP} (g\bullet_{\bsP} f)$. This follows from the computation in \autoref{kleibi}, once we recognize that $\Prof_{\bsP}(\cA)$ is no more no less than the Kleisli bicategory of $\bsP$ regarded as a relative monad.
\end{definition}
\subsubsection{Coend calculus}\label{coendz}
It is natural to wonder whether there is a full coend calculus on $\cA$, induced by the yosegi $\bsP$, and not only a composition rule for internal profunctors. In order to obtain it we have to further refine the structure on the 2\hyp{}category $\cK$ and make its interaction with $\bsP$ more rigid and well-behaved. We have essentially made a further enlargement $\cA\subset \cK \subset \cK'$, in analogy with the change of universe $\caat\subset\Cat\subset\CAT$.

We assume that $\cK'$ is cartesian closed (like it happens for $\CAT$); given an admissible 1\hyp{}cell
$f: A \to B$, the map $\Lan_{y_A}(f)$ can be thought as the formal analogue of $\firstblank \otimes f = \colim_{\firstblank}(f)$, i.e. as the functor $f\mapsto \Lan_{y_A}(f)(W)$. The upshot of this subsection is that when $\cK'$ is cartesian closed, we can internalize the above construction. We shall define a functor
\[\colim_{\firstblank}(\firstblank):\bsP(A) \times [A,B] \to B\]
The Yoneda embedding $y: A \to\bsP(A)$ induces an inverse image $y^*: [P(A), B] \to [A,B]$, whose left adjoint, in case it exists, we call $L$.

Now, the mate of $L$ under the cartesian closure of $\cK$ is precisely the weighted colimit functor, or rather the functor $\colim_{\firstblank}(f)$. Such a left adjoint can't be proved to exist in full generality, when $B$ is cocomplete,  we can hope $y^*$ to be nervous to apply the formal Kan lemma and compute $\Lan_{y^*}(\text{id}_{B^A})$.

In order to recover coend calculus, we shall find a functor
$\int : [A^\dual \times A, B] \stackrel{\int}{\to} B$. To this end,we shall employ $\fh_A : A^\dual \times A \to\bsP(1)$ and its mate, a generic point $\tilde{\fh}_A :1 \to\bsP(A^\dual \times A)$. We can thus define the map
\[
	\textstyle\int: [A^\dual \times A, B] \cong  1 \times [A^\dual \times A, B] \stackrel{h \times \text{id}}{\to}\bsP(A^\dual \times A) \times [A^\dual \times A, B] \stackrel{\colim_{\firstblank}(\firstblank)}{\longrightarrow} B
\]
by composition.
\appendix
%!TEX root = ../yosegi.tex
\section{Relative monads on 2-categories}\label{app:relmo}
Provided the left extension along a given functor $J : \cX\to \cY$ exists, the functor category $[\cX,\cY]$ becomes \emph{skew\hyp{}monoidal} \cite{szlachanyi2012skew} with respect to the functor $\jb : [\cX,\cY]\times[\cX,\cY]\to [\cX,\cY]$ sending $F,G$ to $J_!F\crc G$.%, and has $J$ as a skew unit. 

In the present section we offer a characterization of relative monads as $\jb$-monoids: at the price of lower generality than the treatment in \cite{fiore2016relative}, we can provide a formal analogue for a similar statement than the one given for $\cK = \Cat$ in \cite{altenkirch2010monads}; while extremely useful a reference, the proof in \cite{altenkirch2010monads} does not apparently admit a reformulation in an abstract 2-category. 

In this respect the present appendix has an intrinsic interest, because it provides a full formal proof of existence of the aforementioned skew\hyp{}monoidal structure; although relatively elementary an argument, it revealed to be rather tedious in the choice of notation --and fixing this notation in a separate section allows us to make the exposition look much more streamlined elsewhere.

Our treatment of relative monads is less general than the one in \cite{fiore2016relative}, since in that paper there is no assumption regarding the existence and good behaviour of the left extension functor $J_!$; in full generality, relative monads truly don't have a multiplication map. Fortunately, as observed in our \autoref{not-a-skiu}, `all the structure we need exists locally', since the iterated left extensions of $\bsP$ exist.%, to the point that $j_!\bsP$ is a true monad on $\Cat$.
\begin{defn-prop}\label{it-is-skiu}
Let $J : \cX \to \cY$ be a functor such that the left extension along $J$ exists, defining a functor $J_! : [\cX, \cY]\to [\cY,\cY]$; then the category $[\cX,\cY]$ becomes a \emph{left skew-monoidal category} under the skew multiplication defined by
\[
	(F,G) \mapsto F\jb G = {J_!}F\crc G;
\]
there are natural maps of \emph{association}, \emph{left} and \emph{right unit}
\[\label{coherence-data}
	\begin{tikzcd}[column sep=large, row sep=0mm]
		(F\jb G)\jb H \ar[r, "\gamma_{FGH}"] & F\jb(G\jb H)\\
		J\jb F \ar[r, "\lambda_F"] & F\\
		F \ar[r, "\varrho_F"] & F\jb J
	\end{tikzcd}
\]
such that the following diagrams are commutative:
\begin{enumtag}{skm}
	\item \label{skm:uno}\emph{skew associativity}:
	\[
		\begin{tikzcd}[column sep=2cm]
			((F\jb G)\jb H)\jb K \ar[r,"\gamma_{FG,H,K}"]\ar[d, "\gamma_{F,G,H}\jb K"']& (F\jb G)\jb (H\jb K) \ar[dd, "\gamma_{F,G,HK}"]\\
			(F\jb (G\jb H))\jb K \ar[d, "\gamma_{F,GH,K}"']& {}\\
			F\jb ((G\jb H)\jb K) \ar[r, "F \jb \gamma_{G,H,K}"'] & F\jb (G\jb (H\jb K))
		\end{tikzcd}
	\]
	\item \label{skm:due}\emph{skew left and right unit}:
	\[
		\begin{tikzcd}[column sep=2cm]
			(F\jb G)\jb J\ar[r, "\gamma_{F,G,J}"] \celtag[pos=.1,dr]{\textsc{2r}}& F\jb (G\jb J)\\
			F\jb G\ar[u, "\varrho_{F\jb G}"]\ar[ur, "F\jb \varrho_G"'] \celtag[pos=.9,dr]{\textsc{2l}} & F\jb G\\
			(J\jb F)\jb G \ar[r, "\gamma_{J,F,G}"']\ar[ur, "\lambda_F\jb G"]& J\jb(F\jb G)\ar[u, "\lambda_{F\jb G}"']
		\end{tikzcd}
	\]
	\item \label{skm:tre}\emph{zig\hyp{}zag identity}:
	\[
		\begin{tikzcd}[column sep=small]
			& J\jb J \ar[dr, "\lambda_J"]& \\
			J \ar[ur, "\varrho_J"]\ar[rr,equal]&& J
		\end{tikzcd}
	\]
	\item \label{skm:qtr}\emph{interpolated zig\hyp{}zag identity}:
	\[
		\begin{tikzcd}[column sep=large]
			(F\jb J)\jb G \ar[rr, "\gamma_{F,J,G}"]&& F\jb (J\jb G) \ar[d, "F \jb \lambda_G"]\\
			F\jb G \ar[rr,equal] \ar[u, "\varrho_F\jb G"]&& F \jb G
		\end{tikzcd}
	\]
\end{enumtag}
\end{defn-prop}
The skew monoidal structure, i.e. the natural maps $\gamma,\lambda,\varrho$ is defined as follows:
\begin{enumtag}{s}
	\item \label{s:uno} Given $F,G,H\in [\cX, \cY]$, the cell $\gamma_{F,G,H}$ is defined by $\tilde \gamma_{F,G} * H$, where $\tilde\gamma_{F,G}$ is obtained as the mate of $J_!F * \eta_G$ under the adjunction $J_!\adjunct{\epsilon}{\eta} J^*$: the arrow
	\[\label{associatore} J_!F\crc G \xrightarrow{J_!F * \eta_G} J_!F\crc J^*J_!G  = J^*(J_!F\crc J_!G)\]
	mates to a map $J_!(J_!F\crc G) \longrightarrow J_!F\crc J_!G$.
	\item \label{s:due} Given $F\in [\cX, \cY]$, the cell $\lambda_F : J\jb F \To F$ is defined by the whiskering $\sigma * F$, where $\sigma = \sigma_{1_{\cY}}$ is the counit of the density comonad of $J$;
	\item \label{s:tre} Given $G\in [\cX,\cY]$, the cell $\varrho_G : G \to G\jb J$ is the $G$-component $\eta_G$ of the unit of the adjunction ${J_!} \dashv J^*$
\end{enumtag}
\begin{remark}\label{we-re-better-than-uustalu}
	A complete proof of \autoref{it-is-skiu} appears as Theorem 3.1 in \cite{altenkirch2010monads}; the main argument is however heavily relying on the assumption that $\cK = \Cat$, and thus it is of little use here; we thus produce a formal proof \emph{ex novo}, while explicitly recording some useful equations satisfied by the data defining the skew monoidal structure in study; the structure maps of the skew monoidal structure are entirely induced by the $J_!\adjunct{\eta}{\varepsilon} J^*$ adjunction and from (the equations satisfied by) its unit and counit.

	So, we adopt the following notation, and employ the following equations:
	\begin{enumtag}{e}
		\item \label{e:zer} we denote $\sb{\varpi} : J_!U \to V$ the \emph{mate} of $\varpi : U \to J^*V$ under the adjunction $J_! \adjunct{\eta}{\varepsilon} J^*$; similarly, we denote $\chi^\bullet : U \to J^*V$ the mate of $\chi : J_!U \to V$; in this notation, the bijection
		\[[\cX, \cY](U, J^*V)\cong [\cY, \cY](J_!U, V)\]
		reads as $(\sb{\varpi})^\bullet = \varpi$ e ${}^\bullet(\kern-.1em\chi^\bullet)=\chi$.
		\item \label{e:uno} the first zig\hyp{}zag identity for the adjunction $J_! \adjunct{\eta}{\varepsilon} J^*$ is $(\varepsilon_B * J)\crc \eta_{BJ} = 1_{BJ}$; in particular, if $B = 1_{\cY}$ we have $(\sigma * J)\crc \eta_J = 1_J$;
		\item \label{e:due} the second zig\hyp{}zag identity for the adjunction $J_! \adjunct{\eta}{\varepsilon} J^*$ is $\varepsilon_{J_!F} \crc J_!(\eta_F) =1_{J_!F}$;
		\item \label{e:ter} an irreducible expansion for the associator, expliciting all its components, is
		\[\label{associator}\gamma_{FGH} = \big(\varepsilon_{J_!F\crc J_!G}\crc J_!(J_!F * \eta_G)\big) * H\]
		\item \label{e:qtr} the density comonad of $J$, whose comoltiplication $\nu$ and counit $\sigma$ satisfy the comonad axioms is defined by the maps
		\begin{gather}\sigma = \varepsilon_{1_B} : J_!(J)\to 1 \notag \\ \nu : J_!(J) \xto{J_!(\eta_J)} J_!(J_!(J)\crc J) \xto{\tilde\gamma_{JJ}} J_!(J)\crc J_!(J)
		\end{gather}
		in particular, the comultiplication and the associator determine each other.
	\end{enumtag}
\end{remark}
\begin{proof}
	First of all, axiom \ref{skm:tre} is one of the two triangle identities of the adjunction ${J_!} \dashv J^*$. It remains to prove the other coherence conditions.
	\begin{enumtag}{skm}
		\item It is easy to see that one can prove commutativity before precomposing with $K$; one has then to prove the commutativity of the diagram
		\[
			\begin{tikzcd}[column sep=4cm]
				J_!(J_!(J_!F\crc G)\crc H) \ar[r, "\tilde\gamma_{F\jb G,H}"]\ar[d, "J_!(\tilde\gamma_{FG} * H)"'] & J_!(J_!F\crc G)\crc J_!H \ar[dd, "\tilde\gamma_{FG} * J_!H"]\\
				J_!(J_!F\crc (J_!G\crc H)) \ar[d, "\tilde\gamma_{F, G\jb H}"'] \ar[dr,dashed] \celtag[ur]{\text{\textsc{c}\oldstylenums{1}}}& {}\\
				J_!F \crc J_!(J_!G\crc H) \celtag[near start,ur]{\text{\textsc{c}\oldstylenums{2}}}\ar[r, "J_!F * \tilde\gamma_{GH}"']& J_!F\crc J_!G\crc J_!H
			\end{tikzcd}
		\]
		note that the dashed arrow exists: it is the mate $\sb\alpha$ of $\alpha = (J_!F \crc J_!G) * \eta_H$; the plan is to prove that the two sub\hyp{}diagrams in which this arrow splits the whole diagram commute separately. In order to do this, we start from the square diagram \textsc{c}\oldstylenums{1}: recall that \ref{e:uno} entails that one of the squares
		\[
			\begin{tikzcd}
				J_!A \ar[d, "J_!f"'] \ar[r,"\sb a"]& X\ar[d, "g"] & A \ar[d, "f"'] \ar[r,"a"]& J^* X\ar[d, "J^*g"] \\
				J_!B \ar[r, "\sb b"'] & Y & J_!B \ar[r, "b"'] & J^*Y
			\end{tikzcd}
		\]
		commutes if and only if the other does. Hence, if we denote $f=\tilde\gamma_{FG} * H,g=\tilde\gamma_{FG} * J_!H,a=(J_!F\crc G) * \eta_H,b=\alpha$, \textsc{c}\oldstylenums{1} commutes if and only if the square
		\[
			\begin{tikzcd}[column sep=4cm]
				J_!F \crc G\crc H \ar[d, "\sb(J_!F * \eta_G) * H"']\ar[r, "(J_!F\crc G) * \eta_H"]& J_!(J_!F \crc G)\crc J_!H\crc J\ar[d, "\sb(J_!F * \eta_G) * (J_!H \crc J)"]\\
				J_!F \crc J_!G \crc H \ar[r, "(J_!F\crc J_!G) * \eta_H"']& J_!F\crc J_!G \crc J_!H \crc J
			\end{tikzcd}
		\]
		commutes. It does, since their common value at the diagonal is simply the horizontal composition $\sb{(J_!F * \eta_G)} \myboxmin \eta_H$.

		A similar argument shows that \textsc{c}\oldstylenums{2} commutes: we have to establish the commutativity of
		\[\begin{tikzcd}
				J_!(J_!F\crc J_!G\crc H) \ar[dr, "\sb(J_!F * \eta_{G\jb H})"']\ar[rr, "\sb(J_!F * J_!G * \eta_H)"]&& J_!F\crc J_!G\crc J_!H \\
				& J_!F\crc J_!(J_!G\crc H)\ar[ur, "J_!F * \sb(J_!G * \eta_H)"'] &
			\end{tikzcd}\]
		a diagram that can be `straightened' to the left one below:
		\[
			\begin{tikzcd}[column sep=2cm]
				A \arrow[r, "\sb(J_!F * \eta_{G\jb H})"] \arrow[d,equal] & A \arrow[d, "J_!F * \sb(J_!G * \eta_H)"] && A \ar[d, equal]\ar[r, "J_!F * \eta_{G\jb H}"] & A\ar[d, "J_!F * \sb(J_!G * \eta_H) * J"]\\
				A \arrow[r, "\sb(J_!F * J_!G * \eta_H)"'] & A && A\ar[r, "J_!F * J_!G * \eta_H"'] & A
			\end{tikzcd}
		\] but now, the left diagram commutes if and only if the right one does; and the latter commutativity follows from the definition of $\tilde\gamma$.
		\item A separate argument works for diagrams \textsc{2r} and \textsc{2l}: unwinding the definitions, the commutativity of \textsc{2r} amounts to the commutativity of
		\[
			\begin{tikzcd}
				J_!(J_!F\crc G)\crc J \ar[r, "\gamma_{FGJ}"]\ar[d, "\eta_{J_!F\crc G}"']& J_!F\crc J_!G \crc J \ar[d,equal]\\
				J_!F\crc G \ar[r, "J_!F * \eta_G"]& J_!F \crc J_!G\crc J.
			\end{tikzcd}
		\]
		If we denote for short $J_!F * \eta_G = \varpi$, this commutativity is equivalent to the fact that $J^*(\sb{\varpi})\crc \eta_{J_!F\crc G} = \varpi$, which is true since the left had side of this equation is $(\sb{\varpi})^\bullet$. For axiom \textsc{2l}, the commutativity of
		\[
			\begin{tikzcd}
				J_!(J)\crc F \ar[r, ""] & J^*J_!F\\
				J_!(J)\crc F \ar[u,equal]\ar[r, "J_!F * \eta_F"']& J_!(J)\crc J^*J_!F \ar[u, "\varepsilon * J_!F"']
			\end{tikzcd}
		\] follows from the fact that the upper horizontal row coincides with the horizontal composition $\sigma \myboxmin \eta_F = \eta_F \crc (\sigma * F)$; this means that \textsc{2l} is true if and only if the square
		\[
			\begin{tikzcd}
				F \ar[r,"\eta_F"]& J^*J_!F \\
				J_!J \crc F \ar[u,"\sigma * F"] \ar[r, "J_!F * \eta_F"'] & J_!J \crc J^*J_!F\ar[u, "\varepsilon * J_!F"']
			\end{tikzcd}\]
		commutes; this is obvious by the naturality property of $\eta$.
		\item[\textsc{skm}\oldstylenums{4})] Unwinding the definition, axiom \ref{skm:qtr} asks the diagram
		\[
			\begin{tikzcd}
				J_!F \crc J_!(J) \ar[r, "J_!F * \sigma"]& J_!F \\
				J_!(J_!F\crc J)\ar[u, "\gamma_{FJ}"] & J_!F\ar[u,equal]\ar[l, "J_!(\eta_F)"]
			\end{tikzcd}\]
		to commute. In order to see that it does, we observe that the chain of equivalences
		\begin{align*}
			\underline{(J_!F * \sigma) \crc \varepsilon_{J_!F \crc J_!J}} \crc J_!(J_!F * \eta_J) & = \varepsilon_{J_!F}\crc \underline{J_!(J_!F * \sigma * J) \crc J_!(J_!F * \eta_J)}  \\
 & = \varepsilon_{J_!F} \crc \underline{J_!\big(J_!F * ((\sigma * J) \crc \eta_J)\big)} \\
 & = \varepsilon_{J_!F} \crc 1_{J_!F \crc J}
		\end{align*}
		holds, where the last step is motivated by \ref{e:uno}. So, the composition $(J_!F * \sigma)\crc \tilde\gamma_{FJ}$ equals $\varepsilon_{J_!F}$; this, together with \ref{e:due}, concludes. \qedhere
	\end{enumtag}
\end{proof}
\begin{remark}[A nice $J$ gives a nice structure]\label{adjoints-are-exts}
	In favourable cases the skew monoidal structure simplifies until it collapses to a straight monoidal structure:
	\begin{itemize}
		\item if $J$ is fully faithful, the unit $\eta_F : F \to J_!F \crc J$ is invertible for every $F$, so that $F \cong F \jb J$;
		\item if $J$ is dense, the density comonad $\sigma$ is isomorphic to the identity functor, so $\lambda_F : F \cong J\jb F$ is an isomorphism;
		\item if each extension $J_!F$ is absolute, then each $\tilde\gamma_{F,G}$ is invertible.
	\end{itemize}
\end{remark}
\begin{remark}
		The bifunctoriality of composition has been implicitly employed in the above proof; we spell it out explicitly for future reference.
		
		In order to prove such bifunctoriality, we have to show the commutativity of the square
		\[\begin{tikzcd}
				F\jb G\ar[r, "F\jb g"]\ar[d, "f\jb G"'] & F \jb K\ar[d, "f\jb K"] \\
				H \jb G \ar[r, "H\jb g"']& H\jb K
			\end{tikzcd}\]
	given $f : F\to H$ and $g : G\to K$. In fact, unwinding the definition it's easy to realize that this diagram commutes since its diagonal coincides with the horizontal composition $J_!f\myboxmin g$.
\end{remark}
\begin{remark}[A note on $\jb$-whiskering]\label{on-uisge}
This begs the question of how the whiskering of a 1-cell $F$ with a 2-cell $\alpha$ works, on the left and on the right; given the shape of the $\jb$-skew\hyp{}monoidal structure, it turns out that 
\begin{gather}
\alpha \jb F := J_!\alpha * F\notag \\
F \jb \alpha := J_!F * \alpha
\end{gather}
\end{remark}
\begin{definition}[Monads need not be endofunctors, but are always skew monoids]\label{def:relmo}
We will be interested in the notion of a $J$-\emph{relative monad}, or simply a relative monad when $J$ is understood from the context; $J$-relative monads are defined as internal monoids in the skew-monoidal category $([\cX,\cY], \jb)$, thus they come equipped with a \emph{unit} $\eta : J \To T$ and a multiplication $\mu : T\jb T\To T$. However, since the monoidal structure $\jb$ satisfies pretty asymmetrical coherence conditions, the commutativity conditions satisfied by a relative monad get altered accordingly. In particular,
\begin{enumtag}{rm}
	\item the unit axiom amounts to the commutativity of
	\[\label{relmo:unit}
		\begin{tikzcd}
			J\jb T \arrow[r, "\eta\jb T"] \arrow[rd, "\lambda"'] & T\jb T \arrow[d, "\mu"] & T\jb J \arrow[l, "T\jb \eta"'] \\
			& T \arrow[ru, "\varrho"'] &
		\end{tikzcd}
	\]
	where the right triangle `commutes' in the sense that the composition $\mu\crc (T\jb \eta)\crc \varrho$ makes the identity of $T$.
	\item the multiplication $\mu$ is `skew associative':
	\[\label{relmo:mult}
		\begin{tikzcd}
			(T\jb T)\jb T \arrow[d, "\mu\jb T"'] \arrow[r, "\gamma"] & T\jb(T\jb T) \arrow[r, "T\jb \mu"] & T\jb T \arrow[d, "\mu"] \\
			T\jb T \arrow[rr, "\mu"'] &  & T
		\end{tikzcd}
	\]
\end{enumtag}
\end{definition}
% !TEX root = ../yosegi.tex
\section{Algebras for relative 2-monads}\label{app:contra}
\subsection{The 2-category of algebras of a relative monad}
We recall that
\begin{definition}[Algebra for a total monad]\label{def:alg}
	An \emph{algebra} for $T  : \cC\to \cC$ is a map $a : T A\to A$ such that the diagrams
	\[
		\begin{tikzcd}
			T T A \arrow[r, "\mymu_A"]\arrow[d, "T a"'] & T A \arrow[d, "a"] & A \arrow[r, "\myeta_A"] \arrow[rd,equal] & T A \arrow[d, "a"] \\
			T A \arrow[r, "a"'] & A &  & A
		\end{tikzcd}
	\]
	are filled by isomorphisms.
\end{definition}
\begin{remark}\label{aint-no-algebras}
	In the case of relative monads, things get quite tricky, as very generic relative monads do not seem to possess algebras: indeed, if the coherence $\langle \gamma,\varrho,\lambda\rangle$ of $\bsT$ is highly non\hyp{}invertible the commutative diagrams defining a $\bsT$-algebra do not make sense any more.

	It is indeed quite easy to notice, transporting the $\bsT$-algebra axioms to the relative case, that whatever a $\bsT$\hyp{}algebra is, it must consist of a 1-cell of $\cD$, say $a : TA\to JA$, such that some diagrams `commute'; it is then just as easy to realize that the diagram
	\[
		\begin{tikzcd}
			T\jb TA \arrow[d, "Ta"'] \arrow[r, "\mu_A"] & TA \arrow[dd, "a"] \\
			T\jb JA &  \\
			TA \arrow[u, "\rho_A"] \arrow[r, "a"'] & JA
		\end{tikzcd}
	\]
	can't be made to `commute' unless $\rho$ is invertible. As a rule of thumb, the diagrams produced applying $T$ to the algebra introduces non\hyp{}invertible coherence, and their skewness must be corrected to ensure commutativity.

	If  we define \emph{$\rho$-normal} a relative monad such that $\rho : T \cong T\jb J$ is invertible, the notion of algebra makes sense again for $\rho$-normal monads.
\end{remark}
We also recall
\begin{definition}[Algebra morphism for a total monad]\label{def:laxidem}
	A (lax) morphism of algebras $(A,a)\to (B,b)$ for a total monad $T : \cC \to \cC$ is a pair $(f,\bar f)$ that fills the diagram
	\[
		\begin{tikzcd}
			T A \arrow[r, "T f"]  \arrow[d, "a"']\ar[dr,phantom, "{\scriptstyle\bar f}\Swarrow"] & T B \arrow[d, "b"] \\
			A \arrow[r, "f"'] & B
		\end{tikzcd}
	\]
	These data are subject to the coherence axioms spelled out in \cite[p. 3]{blackwell1989two}, eq. (1.2) and (1.3):
	\begin{gather}
		\begin{tikzcd}[ampersand replacement=\&]
			TTA \arrow[d, "\mu_A"'] \arrow[r, "TTf"] \& TTB \arrow[d, "\mu_B"] \& {} \& TTA \arrow[d, "Ta"'] \arrow[r, "TTf"] \& TTB \arrow[d, "Tb"] \\
			TA \ar[dr,phantom,"{\scriptstyle\bar f}\Swarrow"]\arrow[d, "a"'] \arrow[r, "Tf"] \& TB \arrow[d, "b"] \& \Huge = \& TA \ar[dr,phantom,"{\scriptstyle\bar f}\Swarrow"]\arrow[d, "a"'] \arrow[r, "Tf"] \& TB \arrow[d, "b"] \\
			A \arrow[r, "f"'] \& B \& {} \& A \arrow[r, "f"'] \& B
		\end{tikzcd}\\
		\begin{tikzcd}[ampersand replacement=\&]
			A \arrow[r, "f"] \arrow[d, "\eta_A"'] \& B \arrow[d, "\eta_B"] \& {} \& {} \\
			TA \arrow[d, "a"'] \arrow[r, "Tf"] \& TB \arrow[d, "b"] \& \Huge =  \& \text{id}_f \\
			A \arrow[r, "f"'] \& B \& {} \&
		\end{tikzcd}
	\end{gather}

	A similar definition yields a \emph{colax} morphism of algebras, where the direction of the 2-cell $\bar f$ is reversed.

	\emph{Algebra 2-cells} are defined to be 2-cells $\alpha : f\To g$ in $\cC$  such that the following diagram commutes
	\[\small
		\begin{tikzcd}
			TA \ar[r,phantom,"{\scriptstyle \Downarrow T\alpha}"]\ar[dr,phantom,near end, "{\scriptstyle\bar g\Swarrow}"]\arrow[r, "Tf", bend left] \arrow[r, bend right] \arrow[d, "a"'] & TB \arrow[d, "b"] & TA \ar[dr,phantom,near start, "{\scriptstyle\bar f\Swarrow}"]\arrow[r, "Tf", bend left] \arrow[d, "a"'] & TB \arrow[d, "b"] \\
			A \arrow[r, "g"', bend right] & B & A \ar[r,phantom,"{\scriptstyle\Downarrow\alpha}"]\arrow[r,  bend left] \arrow[r, "g"', bend right] & B
		\end{tikzcd}
	\]
\end{definition}
\begin{remark}
	Again, things get complicated when we relativize the monads in study: the same rule of thumb as above holds, and the diagram whose commutativity breaks now is
	\[
		\begin{tikzcd}
			TTA \ar[drr,phantom, "{\scriptstyle T\bar f}\Swarrow"]\arrow[rr, "TTf"] \arrow[d, "Ta"'] &  & TTB \arrow[d, "Tb"] &  & TTA \arrow[rr, "Tf"] \arrow[dd, "\mu_A"'] &  & TTB \arrow[dd, "\mu_B"] \\
			TJA \arrow[rr] &  & TJB &  &  &  &  \\
			TA \ar[drr,phantom, "{\scriptstyle\bar f}\Swarrow"]\arrow[u, "\rho_A"] \arrow[d, "a"'] \arrow[rr] &  & TB \arrow[u, "\rho_B"'] \arrow[d, "b"] & \Huge = & TA\ar[drr,phantom,"{\scriptstyle\bar f}\Swarrow"] \arrow[d] \arrow[rr, "Tf"] &  & TB \arrow[d] \\
			JA \arrow[rr, "Jf"'] &  & JB &  & JA \arrow[rr, "Jf"'] &  & JB
		\end{tikzcd}
	\]
unless $\bsT$ is $\rho$-normal there is no way we can give meaning to this equality.

Note that instead, the unit condition and the notion of algebra 2-cell are fortunately unaffected by relativization of $\bsT$.
\end{remark}
\begin{notation}
Each time we speak about the 2-category of algebras for a relative monad $\bsT$, we implicitly assume that $\bsT$ is $\rho$-normal; obviously this assumption does not affect the discussion about yosegi boxes and Yoneda structures, since the underlying monad of a yosegi is always normal.
\end{notation}
\subsection{Algebras in the lax idempotent case}
\begin{proposition}\label{lax-equivs}
	Lax idempotent monads admit equivalent definitions as follows; the colax version of each condition is stated beside the correspondent lax version.
	\begin{enumtag}{l}
		\item \label{l:uno} for every pair of $ {\bsT}$-algebras $a,b$ and morphism $f :A \to B$, the square
		\[\notag
			\begin{tikzcd}[row sep=8mm,column sep=8mm]
				{\bsT} A \arrow[d, "a"']\ar[dr, phantom, "\Swarrow"] \ar[r,"{\bsT}f"] &  {\bsT} B \arrow[d, "b"]\\
				A \arrow[r, "f"'] & B
			\end{tikzcd}
		\]
		is filled by a unique 2-cell $\bar f : b \crc \bsT f\To f\crc a $ which is a lax morphism of algebras;
		\item \label{l:due} $a\dashv \myeta_A$ with invertible counit;
		\item \label{l:ter} $\mymu\dashv \myeta * {\bsT}$ with invertible counit;
		\item there is a modification $\Delta : {\bsT} * \myeta \To \myeta * {\bsT}$ such that $\Delta * \myeta = 1$ and $\mymu * \Delta = 1$.
	\end{enumtag}
	The conditions for colax algebras are of course obtained replacing lax algebra morphisms with \emph{colax} ones.
	\begin{enumtag}{c}
		\item \label{c:uno} for every pair of $ {\bsT}$-algebras $a,b$ and morphism $f :A \to B$, the square
		\[\notag
			\begin{tikzcd}[row sep=8mm,column sep=8mm]
				{\bsT} A \arrow[d, "a"']\ar[dr, phantom, "\Nearrow"] \ar[r,"{\bsT}f"] &  {\bsT} B \arrow[d, "b"]\\
				A \arrow[r, "f"'] & B
			\end{tikzcd}
		\]
		is filled by a unique 2-cell $\bar f : f\crc a \To b \crc \bsT f$ which is a colax morphism of algebras;
		\item \label{c:due} $\myeta_A \dashv a$ with invertible unit;
		\item \label{c:ter} $\myeta  *  {\bsT}\dashv \mymu$ with invertible unit;
		\item there is a modification $\Upsilon : \myeta * {\bsT} \To {\bsT} * \myeta$ such that $\Upsilon * \myeta = 1$ and $\mymu * \Upsilon = 1$.
	\end{enumtag}
\end{proposition}
Total KZ\hyp{}monads have a very nice and well behaved notion of cocomplete object with respect to them. The notion is identical to the one of cocomplete object with respect to a Yoneda structure as described in \ref{yoda-cocompleta}. When $\bsP$ is a relative KZ-monad every object of the form $\bsP A$ for some $A$ is cocomplete, as observed by \cite[Rem. 5]{walker} in the non-relative case. This observation is still true in the relative case.

\begin{remark}[Relative of \protect{\cite[Rem. 5 \& Prop. 6]{walker}}]\label{KZ-cocomplete}\leavevmode
When $\bsP$ is a relative KZ-monad every object of the form $\bsP A$ for some $A$ is cocomplete, moreover pseudo-agebras coincides with cocomplete objects. 
\end{remark}

\subsection{The Kleisli bicategory of a relative monad}
As is well known, every total monad $T : \cC\to \cC$ has a Kleisli category with the same objects of $\cC$ and where morphisms are 1-cells $f : A \to TB$. Composition is given by the rule
\[\label{kle-comp}
	g \bullet f := \mu_C \crc Tg \crc f,
\]
associative and unital (with identity $\eta_A : A\to TA$ the unit map of the monad) thanks to the monad axioms.

The same structure exists in 2-dimensional algebra, so that each 2-monad has a Kleisli 2-category. In the relative case, things get a little bit hairy: the maximal possible generality is the one of \cite[§5]{fiore2016relative}, but in that case we can't make sense to a `composition law' in the same sense of \eqref{kle-comp}, as there is no real multiplication for a relative monad.

Mildly restricting the class of relative monads in study (more or less, those for which enough left extensions exist in order to ensure the existence of $\mymu$) we obtain a formalism in which the result that the Kleisli bicategory of a yosegi is a Yoneda equipment appears natural and well\hyp{}motivated.
\begin{definition}[Kleisli bicategory of a relative monad]\label{kleibi}
	Let $\bsT$ be a relative monad with underlying functor $T : \cC\to \cD$; then, in case $\bsT$ is normal (i.e., both $\rho$-normal and $\lambda$-normal) we define its \emph{Kleisli bicategory} posing:
	\begin{itemize}
		\item 0-cells the objects of $\cC$;
		\item 1-cells $f : X\pto Y$ the 1-cells $f : JX \to TY$;
		\item composition of 1-cells $X\overset{f}{\pto}Y\overset{g}{\pto}Z$ defined by the composition
		      \[\begin{tikzcd}
				      JX \ar[drr,phantom,"\overset{g\bullet f}{\pto}"]\arrow[d, "f"'] &  & TZ \\
				      TY \arrow[r, "\rho_Y"'] & T\jb JY \arrow[r, "P\jb g"'] & T\jb TZ \arrow[u, "\mu_Z"']
			      \end{tikzcd}\]
		\item all 2-cells between 1-cells.
	\end{itemize}
\end{definition}
In order to show that this is really a bicategory we shall show that the unit of the monad plays the role of two-sided identity (up to an isomorphism prescribed by the monad structure), and that the composition is associative (up to an isomorphism prescribed by the monad structure).
\begin{itemize}
	\item $\eta_Y \bullet f = \mu_Y\crc T\jb\eta_Y \crc \rho_Y\crc f = f$ since \eqref{relmo:unit} holds.
	\item for what concerns the composition $f\bullet \eta_X$ we can reduce it to
	      \begin{align*}
		      f\bullet \eta_X & = \mu_Y\crc T\jb f\crc \rho_X\crc \eta_X              \\
		                      & = \mu_Y \crc \eta\jb TY \crc \lambda_Y^{-1}\crc f = f
	      \end{align*}
	      using the fact that $\bsT$ is normal.
	\item Associativity amounts to verify that given $W\overset{f}{\pto} Z\overset{g}{\pto} X \overset{h}{\pto} Y$ the following two lines are equal:
	      \begin{align*}
		      ( h \bullet g)\bullet f & = \mu_Y \crc T\jb(\mu_Y\crc T\jb h \crc \rho_X\crc g)\crc\rho_Z\crc f                              \\
		                              & = \mu_Y \crc \underline{T\jb\mu_Y\crc T\jb T\jb T\jb h \crc T\jb\rho_X}\crc T\jb g\crc\rho_Z\crc f \\
		      h\bullet(g\bullet f)    & = \mu_Y \crc \underline{T\jb h \crc \rho_X \crc \mu_X} \crc T\jb g\crc \rho_Z\crc f.
	      \end{align*}
	      This follows from the fact that $\mu_Y$ coequalizes the pair of underlined arrows, which in turn follows from the associativity axiom for $\mu$, in \eqref{relmo:mult}.
\end{itemize}

% \paragraph{\bf Acknowledgements}
% The first author is supported by grants GA17\hyp{}27844S and \newline MUNI\fshyp{}A\fshyp{}1103\fshyp{}2017. The second author is supported by the Max Planck Inst. for Math. (Bonn), and wrote the present paper, in its entirety, during his stay at the Institute.

\footnotesize
\bibliography{allofthem.bib}
\bibliographystyle{alpha}
\end{document}